\documentclass[letterpaper,11pt,leqno]{amsart}
\usepackage[latin1]{inputenc}
\usepackage{latexsym}
\usepackage[dvips]{graphicx}
\usepackage{amsfonts}
\usepackage{epsfig}
\usepackage{geometry}
\usepackage{color}
\usepackage{amsmath}
\usepackage{amssymb}
 \usepackage{pict2e}

\newtheorem{ejem}{Example}

\newcommand{\N}{\mathbb{N}}

\newcommand{\R}{\mathbb{R}}
\renewcommand{\P}{\mathcal{P}}

\newcommand{\C}{\mathcal{C}}

\renewcommand{\H}{\mathcal{H}}

\renewcommand{\S}{\mathcal{S}}

\newtheorem{theorem}{Theorem}[section]
\newtheorem{lemma}[theorem]{Lemma}

\newtheorem{remark}[theorem]{Remark}
\newtheorem{proposition}[theorem]{Proposition}

\newtheorem{definition}[theorem]{Definition}
\begin{document}

\title{A dynamic gradient approach to Pareto optimization with nonsmooth convex objective functions}

\author{H\'edy Attouch}

\author{Guillaume Garrigos}

\author{Xavier Goudou}

\address{Institut de Math\'ematiques et Mod\'elisation de Montpellier, UMR 5149 CNRS, Universit\'e Montpellier 2, place Eug\`ene Bataillon,
34095 Montpellier cedex 5, France}
\email{hedy.attouch@univ-montp2.fr, guillaume.garrigos@gmail.com, xavier.goudou@yahoo.fr}

\vspace{0.5cm}

\date{June 5, 2014}

\begin{abstract}

In a general Hilbert framework, we consider  continuous gradient-like dynamical systems for constrained multiobjective optimization involving non-smooth convex objective functions.
Based on the Yosida regularization of the  subdifferential operators involved in the system, we  obtain the  existence of strong global trajectories.   We prove a  descent property for each objective function, and the  convergence of trajectories to weak Pareto minima. This approach provides a dynamical endogenous weighting of the objective functions.
Applications are given to cooperative games, inverse problems, and numerical multiobjective optimization.

\end{abstract}

\maketitle

\vspace{0.5cm}


\vspace{1cm}


\markboth{H. ATTOUCH, G. GARRIGOS, X. GOUDOU}
  {Dynamical Pareto-optimization with nonsmooth convex objective functions}

\maketitle

\section*{Introduction}

Throughout the paper, $\mathcal H$ is a real Hilbert space with scalar product and norm denoted by  $\left\langle  \cdot, \cdot \right\rangle $ and 
$\|\cdot\| = \sqrt{\left\langle  \cdot,\cdot \right\rangle} $  respectively.
We are interested with a gradient-like dynamical approach to the Pareto optima of the Constrained Multiobjective Optimization problem ((CMO) for short)
\begin{equation*}
{\rm (CMO)} \quad \min \left\{ F(v): \  v\in K \right\}
\end{equation*}
where $F: \mathcal H \rightarrow {\mathbb R}^q$,  \   $F(v)= \left(f_i (v) \right)_{i=1,...,q}$, $q\in \mathbb N^*$. 
Working in a general Hilbert space (possibly infinite dimensional) covers both applications in decision sciences and engineering.
We make the following standing assumptions on the  multiple objective functions $(f_i)_{i=1,2,..,q}$,   and constraint $K$:

\smallskip

H0) $K\subset \mathcal H$ is a closed convex nonempty set. 

\smallskip

\noindent For each  \ $i=1,2,...,q$,  \ $f_i: \mathcal H \rightarrow \mathbb R$ is a real-valued function  which satisfies:

\smallskip

H1)  $f_i$  is  convex continuous. It is supposed to be Lipschitz continuous on bounded sets. Equivalently, its subdifferential $\partial f_i : \mathcal H \rightarrow  2^{\mathcal H}$ is bounded on bounded sets;

\smallskip

H2)  $f_i$ is bounded from below on $\H$.

\smallskip

We are interested in this paper with the \textit{lazy} solutions (also called slow solutions, see \cite[Ch. 6, section 8]{AE}) of  the  differential inclusion
\begin{equation}\label{mult-steep-desc2}
 \dot u(t) + N_K(u(t)) + \mbox{Conv}\left\{\partial f_i(u(t))\right\} \ni 0,\\
\end{equation}
which is governed by  the  sum  of the two set-valued operators $u \mapsto N_K(u)$, and  $u\mapsto \mbox{Conv}\left\{\partial f_i(u)\right\} $. 
For $u \in K$, $N_K(u)$ is the  normal cone to $K$ at $u$, a closed convex cone modeling the \textit{contact forces} which are attached to the constraint $K$. Besides, $\mbox{Conv}\left\{\partial f_i(u)\right\}$ denotes the closed convex hull of the sets $\left\{\partial f_i(u); \ i=1,...,q\right\}$, and models the \textit{driving forces} which govern our system. 

\noindent Lazy solution means that the trajectory chooses a velocity which has minimal norm among all possible directions offered by the differential inclusion. 
This type of differential inclusion  occurs in various domains (mechanics, economics, control...), and has subsequently be the object of active research, see for example  \cite{AD}, \cite{AC}, \cite{BCS}, \cite{CS}.
Precisely, for any $u\in K$, the set $-N_K(u) -\mbox{Conv}\left\{\partial f_i(u) \right\}$
 is a closed convex set, therefore it has a unique element of minimal norm, denoted as usual $ \left(-N_K(u) - \mbox{Conv}\left\{\partial f_i(u) \right\} \right)^{0} $. 
The direction $s(u):=\left(-N_K(u) - \mbox{Conv}\left\{\partial f_i(u) \right\} \right)^{0}$ is  called the \textit{multiobjective steepest descent direction} at $u$, and the associated dynamical system 
\begin{equation}\label{mult-steep-desc}
\mbox{{\rm (MOG)}} \quad   \dot u(t) + \bigg(N_K(u(t)) + \mbox{Conv}\left\{\partial f_i(u(t)) \right\} \bigg)^{0}= 0,
\end{equation}
is called the  Multi-Objective Gradient system, (MOG) for short.
It was first investigated  by  Henry \cite{Hen},  Cornet  \cite{Corn1}, \cite{Corn2}, \cite{Corn3}, and Smale \cite{Sm1} in the seventies (in the finite dimensional case, 
and in the case where the objective functions $f_i$ are smooth), as a dynamical model of allocation of resources in  economics (planification procedure).

\if{
\noindent The description of our dynamic system involves the following convex sets in $\mathcal H$: \ for $u \in K$, 
$$
N_K(u) = \{ z\in \mathcal H: \ \left\langle  z,v -u \right\rangle \leq 0 \quad \mbox{for any} \ v \in K \}
$$
is the  normal cone to $K$ at $u$. It is a closed convex cone which models the \textit{contact forces} which are attached to the constraint $K$. 

To develop gradient-like methods involving nondifferentiable convex objective functions, we use the  classical notion of subdifferential. Recall that,  for a convex continuous function $f: \mathcal H \rightarrow \mathbb R$, the subdifferential of $f$ at  $u\in \mathcal H$, is the nonempty closed convex and bounded set
\begin{equation}
\label{e:subdiff}
\partial f (u)= \left\{ z \in  \mathcal H :      f(\xi) \geq f(u) + \left\langle  z  ,  \xi -u  \right\rangle  \quad\forall \xi \in \mathcal H \right\}.
\end{equation}
When $f$ is differentiable at $u$, it coincides with the gradient of $f$ at $u$.\\
Denote by $\mathcal{S}^q=\{ \theta=(\theta_i) \in \R^q : 0 \leq \theta_i \leq 1, \ \sum_{i=1}^{q} \theta_i =1 \}$  the unit simplex  in $\R^q$.  
Then, for $u \in \mathcal H$, 
 $$
\mbox{Conv}\left\{\partial f_i(u); \ i=1,...,q \right\} = \{ \sum_{i=1}^{q} \theta_i   \xi_i  : \ {\xi}_i \in \partial f_i(u), \ (\theta_i)\in \S^q \}
$$
is the closed convex hull of the sets $\left\{\partial f_i(u); \ i=1,...,q\right\}$. It models the \textit{driving forces} which govern our system. 
To simplify the notation, we denote this set by $\mbox{Conv}\left\{\partial f_i(u)\right\}$, or in a more concise way by $\mathcal{C} (u)$.
For any closed convex set $C\subset \mathcal H$ we denote by
$$
C^{0} = \mbox{proj}_C 0
$$
the projection of the origin onto $C$, which is the element of minimal norm of $C$.

We are interested in this paper with the \textit{lazy} solutions (also called slow solutions, see \cite[Ch. 6, section 8]{AE}) of  the  differential inclusion
\begin{equation}\label{mult-steep-desc2}
 \dot u(t) + N_K(u(t)) + \mbox{Conv}\left\{\partial f_i(u(t))\right\} \ni 0,\\
\end{equation}
which is governed by  the  sum  of the maximal monotone operator $ N_K$, and the multivalued upper semicontinuous operator  $v\mapsto \mbox{Conv}\left\{\partial f_i(v)\right\} $. 
Lazy solution means that the trajectory chooses a velocity which has minimal norm among all possible directions offered by the differential inclusion. 
This type of differential inclusion  occurs in various domains (mechanics, economics, control...), and has subsequently be the object of active research, see for example  \cite{AD}, \cite{AC}, \cite{BCS}, \cite{CS}.
Precisely, for any $u\in K$, the set $-N_K(u) -\mbox{Conv}\left\{\partial f_i(u) \right\}$
 is a closed convex set. It has a unique element of minimal norm, $s(u):= \left(-N_K(u) - \mbox{Conv}\left\{\partial f_i(u) \right\} \right)^{0} $,
 called the \textit{multiobjective steepest descent direction} at $u$.
The dynamical system which is governed by the vector field $u\mapsto s(u) $
\begin{equation}\label{mult-steep-desc}
\mbox{{\rm (MOG)}} \quad   \dot u(t) + \bigg(N_K(u(t)) + \mbox{Conv}\left\{\partial f_i(u(t)) \right\} \bigg)^{0}= 0,
\end{equation}
is called the  Multi-Objective Continuous Steepest descent, (MOG) for short.
It was first investigated  by  Henry \cite{Hen},  Cornet  \cite{Corn1}, \cite{Corn2}, \cite{Corn3}, and Smale \cite{Sm1} in the seventies (in the finite dimensional case, 
and in the case where the objective functions $f_i$ are smooth), as a dynamical model of allocation of resources in  economics (planification procedure).
}\fi

\noindent  From the point of view of modeling, we will show that the (MOG) system has the  following  properties:

a) It is a descent method, i.e., for \textit{each} $i=1,...,q$,  \   $t \mapsto f_i(u(t))$ is  nonincreasing.

b) Its trajectories converge to weak Pareto optimal points.

c) The scalarization of the multiobjective optimization problem is done endogeneously. At time $t$, the vector field which governs the system involves a convex combination $ \sum_{i=1}^{q} \theta_i (t) \partial f_i(\cdot)$ of the subdifferential  $\partial f_i(\cdot)$, with scalars $\theta_i (t)$ which 
are not fixed in advance. They are part of the process, whence the decentralized features of this dynamic. 

  Indeed, this system  provides a weighting of the different criteria which offers applications
in various domains, and which are still largely to explore. In inverse problems, signal/imaging processing,  putting  convenient weights on the data fitting term, and the regularization, or sparsity term is a central question.  In game theory, economics, social science, management, the multiobjective steepest descent direction has attractive properties: it improves each of the objective functions, while putting a higher weight on the \textquotedblleft weakest\textquotedblright  agents, a key property of the interaction between cooperating agents. 
In addition, for each trajectory, the Pareto  equilibrium which is finally reached, is not too far from the initial Cauchy data. In some particular situations, it is the projection of the initial data on the Pareto set.

Mathematical analysis of (MOG) gives rise to general statements whose formulation is simple, but some proofs are quite technical. The dynamic is governed by a vector field, $u\mapsto s(u)$, which is discontinuous. Indeed, the  multivalued  operators $ u \mapsto N_K(u)$ and $u\mapsto \mbox{Conv}\left\{\partial f_i(u)\right\}$ are only upper semicontinuous (with closed graphs). Moreover, in general,  $u\mapsto s(u)$ is not a gradient vector field, nor Lipschitz continuous, and $u\mapsto -s(u)$ is not a monotone operator. 
Let us list our main results concerning the  (MOG) dynamical system. A section is devoted to each of them.
 
i) In {Theorem} \ref{steep;desc;dir}, section \ref{S:Pareto},  based on von Neumann's minmax duality theorem, we provide   equivalent formulations of the multiobjective
 steepest descent direction: instead of the subdifferential operators, they make use of the directional derivatives of the objective functions.

ii) In Theorem \ref{asymp1}, section \ref{asp},  we prove the weak convergence of the trajectories of  (MOG) to weak Pareto optimal  solutions of the  constrained multiobjective optimization problem (CMO). 
Our proof is in the line of the proof of convergence of the steepest descent by Goudou and Munier \cite{GM} in the case of a single (quasi-convex) objective function; it makes use of Lyapunov analysis and  Opial's lemma.

iii) In Theorem \ref{basic-exist}, section \ref{ex-yo}, assuming further that $\mathcal H$ is finite dimensional,  we prove the existence of strong global solutions to (MOG).
This is the more delicate part of the mathematical analysis.
We provide a constructive proof which is based on the 
 Yosida approximation of $\partial f_i$, and Peano existence theorem for differential equations. 
 The regularized equations are  relevant of the existence results which have been obtained in \cite{AG1}.
 The difficult point is to pass to the limit on the regularized differential equations, as the regularization parameter goes to zero, because the vector field which governs our dynamic is not continuous, nor monotone. 

iv) In section \ref{mod}, some  modeling and numerical aspects  are discussed for the (MOG) system. We first consider some connections between (MOG) and modeling in cooperative games, and inverse problems  (signal/imaging processing). Then,  by time discretization of  (MOG), we introduce numerical algorithms for nonsmooth constrained multiobjective optimization, and make the link with the recent studies  of Fliege and Svaiter \cite{FS}, Grana Drummond and  Svaiter \cite{DS},  Bonnel, Iusem and Svaiter \cite{BIS}.

We end with a conclusion and some perspectives.


\section{Pareto optimality and multiobjective steepest descent}\label{S:Pareto}

As a preliminary, let us make precise some classical notions of variational analysis. Given a proper lower-semicontinuous convex function $f : \H \longrightarrow \R \cup \{+\infty\}$, its \textit{subdifferential} $\partial f(u)$ is the closed convex subset of $\H$  defined for any $u \in \H$ by
\begin{equation*}
\label{e:subdiff}
\partial f (u)= \left\{ p \in  \mathcal H :      f(v) \geq f(u) + \left\langle  p  ,  v -u  \right\rangle  \quad\forall v \in \mathcal H \right\}.
\end{equation*}
In the special case where $f=\delta_K$ is the indicator function  of a nonempty closed convex set $K \subset \H$,
 the subdifferential of $\delta_K$ at $u\in K$ is the \textit{normal cone} to $K$ at $u$, denoted $N_K(u)$. The subdifferential operator enjoys the following additivity rule: 
let $f$ and $g$ be two proper lower-semicontinuous convex functions such that one of them is continuous at a point belonging to the domain of the other, then  
$$ \partial (f+g) (v) = \partial f(v) + \partial g(v) \ \text{ for all } v \in \H.$$
Suppose now that $f$ is locally Lipschitz continuous. Then the subdifferential of $f$ at  $u\in \mathcal H$ is a nonempty closed convex and bounded set. It is also interesting to consider the \textit{directional derivative} of $f$ at $u\in\H$ in the direction $d\in\H$, defined by 
$$df(u,d)=\ \underset{t \downarrow 0}{\text{\rm lim}} \ \frac{f(u+td) - f(u)}{t}.$$
\noindent
 It is in duality with the subdifferential since $df(u,d)$ is equal to $\sup \{  \langle p,d \rangle  :  p \in \partial f(u) \}$.
Thus, in our context, this directional derivative takes only finite values. Furthermore, for all $u\in \H$, $df(u, \cdot )$ is convex and Lipschitz continuous on $\H$.

\subsection{Pareto optimality} 

When considering problem (CMO), which is to minimize
various cost functions on $K$, we seek a solution in the sense of Pareto, i.e., none of the objective functions can be improved in value without degrading some of the other objective values.
It is a cooperative approach, the  mathematical formulation  is described below.
 
\begin{definition}\label{Pareto2} (Pareto optimality)

i)  An element $ u \in K$  is called Pareto
optimal if  there does not exist  $v \in K$  such that $f_i(v) \leq f_i(u)$ for all $i=1,...,q$, and $f_j(v) < f_j(u)$ for one $j\in {1,...,q}$.

ii) An element $ u \in K$  is called weak Pareto 
optimal if   there  does not exist  $v \in K$  such that $f_i(v) < f_i(u)$ for all $i=1,...,q$.
\end{definition}
\noindent We equip  $\R^q$ with the order $y \preceq z \Leftrightarrow y_i \leq z_i$  \ for all    \ $i=1,...,q$,  and the strict order relation $y \prec z \Leftrightarrow y_i < z_i$  \ for all    \ $i=1,...,q$.
Then Pareto optimality admits an equivalent formulation: 
$ u \in K$ is Pareto optimal iff there  does not exist
$v \in K$  such that $F(v) \preceq  F(u)$ and $F(v) \neq F(u)$, while $u$ is a weak Pareto optimum if there  does not exist $v \in K$  such that $F(v) \prec  F(u)$. These optimality notions can be generalized by considering orders generated by convex cones (see \cite{Din}). 

In the case of a single objective function $f$, Pareto and weak Pareto optima  coincide with the notion of global minimizer. Writing the necessary optimality condition leads to the notion of critical point, namely $ N_K({u}) + \partial f(u) \ni 0$.
A similar approach exists for Pareto optimality: 

\begin{definition}\label{Pareto}
Denote by $\mathcal{S}^q=\{ \theta=(\theta_i) \in \R^q : 0 \leq \theta_i \leq 1, \ \sum_{i=1}^{q} \theta_i =1 \}$  the \textit{unit simplex}  in $\R^q$.
 We say that $u \in K$ is a  Pareto critical point of the constrained multiobjective optimization problem {\rm(CMO)}  
 if there exists $\left(\theta_{i} \right)\in \S^q$  such that 
\begin{equation}\label{asymp4} 
	 N_K({u}) + 	\sum_{i=1}^q \theta_{i}\partial f_i({u})  \ni 0.
\end{equation} 
\end{definition}

\noindent In the differentiable case, this notion has been considered by Smale in \cite{Sm1}, Cornet in \cite{Corn1}, \ see \cite{BM}, \cite{DS}, \cite{Din}   for recent account of this notion, and various extensions of it. It is a multiobjective extension of the Fermat rule, and (see below) a first-order necessary optimality condition for (local) vectorial optimization. 
Note that equivalent formulations of this notion can be given, thanks to the positive homogeneity property of the formula: the condition $\sum_{i=1}^q \theta_{i} =1$ can be dropped, just assuming the $\theta_{i}$ to be nonnegative, and at least one of them positive. 

Let us respectively denote by $\mathcal P$,   ${\mathcal P}_w$, and  ${\mathcal P}_c$ the set of Pareto optima, weak Pareto optima, and Pareto critical points. Clearly, $\mathcal P \subset {\mathcal P}_w$ always holds. In the case of convex multi-objective optimization, critical Pareto optimality is a necessary \textit{and sufficient} condition for weak Pareto optimality. Let us state it in a precise way.

 \begin{lemma}\label{equiv} 
Let $f_i: \mathcal H \rightarrow \mathbb R$, \ $i=1,...,q$  be   convex objective functions. Then 
 ${\mathcal P}_w  = {\mathcal P}_c$. Assuming further that the objective functions are strictly convex, then all these concepts of Pareto optimality coincide, i.e.,
  ${\mathcal P}={\mathcal P}_w={\mathcal P}_c$.
 \end{lemma}
 \begin{proof} 
 The inclusion ${\mathcal P}_w \subset{\mathcal P}_c$ is  obtained for convex functions by a direct application of the Hahn-Banach separation theorem (see for example \cite[Proposition 1.1]{Corn1},   \cite{Din}, \cite{BM}).
 Let us prove the reverse inclusion.
  Let $u \in {\mathcal P}_c$.  Then, $u$ is a (global) solution of the convex minimization problem 
 \begin{equation}\label{min-conv} 
	\min \left\{  \ 	\sum_{i=1}^q \theta_{i}f_i(v): \ v\in K\right\}
\end{equation}
for some $\theta_{i} \in [0,1]$ which are all  nonnegative, and at least one of them positive.
Indeed, (\ref{min-conv}) forces $u$ to be a weak Pareto minimum. Otherwise, there would exist some $v \in K$  such that $f_i(v) < f_i(u)$ for all $i=1,...,q$, which would  imply 
(one uses the  fact that at least one of the $\theta_{i}$ is positive)
 	$\sum_{i=1}^q \theta_{i} f_i(v) < 	\sum_{i=1}^q \theta_{i}f_i(u)$, a clear contradiction.   Now suppose that the objective functions are strictly convex, 
 	then  ${\sum\limits_{i=1}^{q} \theta_i f_i}$ is also strictly convex (we use again the  fact that at least one of the $\theta_{i}$ is positive), and $u$ is its unique minimizer over $K$.
 	 If we assume the existence of  $v \in K$ such that $f_i(v) \leq f_i(u)$ for all $i=1,...,q$, this would imply that $v$ is also a minimizer of ${\sum\limits_{i=1}^{q} \theta_i f_i}$ over $K$.
 	  Hence $v=u$, and $u$ is  Pareto optimal.
 \end{proof}
 
\subsection{Multiobjective steepest descent direction}

{

We discuss the concept of multiobjective descent direction, and present a multiobjective steepest descent direction, by analogy with the case of a single criterion.
\noindent Considering the problem of minimizing a single objective $f$ over $K$, we say that $d$ is a descent direction at $u\in K$ when $df(u,d) <0$, and $d$ lies in the closed convex \textit{tangent cone} to $K$ at $u$, which is defined as the \textit{polar cone} of $N_K(u)$ 
$$T_K(u) : = \{ v \in \H : \langle v, \eta \rangle \leq 0 \ \text{ for all } \eta \in N_K(u) \}.$$

\noindent Following Smale \cite{Sm1}, let us generalize this notion of descent direction for the multiobjective optimization problem (CMO).

\begin{definition}
Considering the problem (CMO), we say that $d \in \H$ is a multiobjective descent direction at $u\in K$ if $\ df_i(u,d)<0$ for each $i\in\{1,...,q\}$,
 and $d \in T_K(u)$.
\end{definition}

\begin{remark}
{\em Define the closed convex hull of the subdifferentials at $u\in \H$ $$\mbox{Conv}\left\{\partial f_i(u) ; \ i=1,...,q \right\} := \{ \sum_{i=1}^{q} \theta_i   p_i  : \ {p}_i \in \partial f_i(u), \ (\theta_i)\in \S^q \}.$$ To simplify the notation we just write $\mbox{Conv}\left\{\partial f_i(u)  \right\}$. Then, from the dual characterization $df_i(u,d) = \sup \{  \langle p,d \rangle  :  p \in \partial f(u) \}$ and the definition of $T_K(u)$, $d$ is a multiobjective descent direction at $u$ iff
$$ \langle p,d \rangle <0 \text{ and } \langle \eta , d \rangle \leq 0 \ \  \forall p \in \mbox{Conv}\left\{\partial f_i(u)  \right\}, \forall \eta \in N_K(u).$$
It is therefore clear that no multiobjective descent direction can be found at a critical Pareto $u \in \mathcal{P}_c$, since this is equivalent to $0 \in N_K(u) + \mbox{Conv}\left\{\partial f_i(u) \right\}$.}
\end{remark}

Let us define the vector field that governs our dynamical system.
\begin{definition}\label{mstdd}
For any $u\in K$, the unique element of minimal norm of the  closed convex set  $-N_K(u) -\mbox{Conv}\left\{\partial f_i(u) \right\}$
is called the \textit{multiobjective steepest descent direction} at $u$. It is denoted by 
 \begin{equation}\label{mult-steep-direction}
 s(u) := \bigg(-N_K(u) - \mbox{Conv}\left\{\partial f_i(u) \right\} \bigg)^{0}.
 \end{equation}
\end{definition}

\noindent Note that, for any $u\in K$, the set $-N_K(u) -\mbox{Conv}\left\{\partial f_i(u) \right\}$
 is a closed convex set, as being equal to the vectorial sum of two closed convex sets, one of them being bounded.
Hence,  it has a unique element of minimal norm, and $s(u)$ is well defined. See Theorem \ref{steep;desc;dir} for an equivalent formulation of the multiobjective steepest descent direction that makes use of dual notions, namely the directional derivatives of the objective functions, and  the tangent cone to $K$.
This vector field  clearly satisfies  $u \in \mathcal{P}_c \Leftrightarrow s(u)=0$.  Furthermore, for any  $u \notin \mathcal{P}_c$, $s(u)$ is a multiobjective descent direction:

\begin{proposition}\label{P:MultiSteepestDescent}For all $u\in K$ we have
\begin{equation}\label{DescentProperty}
\langle s(u),p\rangle \leq - \Vert s(u) \Vert^2 \ \text{ for all } p \in \mbox{Conv}\left\{\partial f_i(u) \right\}.
\end{equation} 
In particular, $s(u)$ is a multiobjective descent direction at any $u\in K \setminus \mathcal{P}_c$.
\end{proposition}

\noindent {\rm The above result has a simple geometrical interpretation. Take for simplicity the unconstrained problem, i.e., $K= \mathcal H$ and two criteria $f_1$, $f_2$.
Then  $-s(u)$ is the orthogonal projection of the origin on the vectorial segment $[\nabla f_1(u),\nabla f_2(u)]$. By the classical result on the sum of the angles of a triangle, this forces 
the angles bewteen $-s(u)$ and $\nabla f_i(u)$, $i=1,2$ to be accute, and hence $\left\langle  s(u), \nabla f_i(u) \right\rangle \leq 0$.}
In order to prove Proposition \ref{P:MultiSteepestDescent}, and in the following, we will make frequent use of the  Moreau decomposition theorem \cite{Mo}.

\begin{theorem} (Moreau)\label{Moreau} 
Let $T$ be a closed convex cone of a real Hilbert space $\mathcal H$, and $N$ be its polar cone, i.e., 
$N= \left\{v \in \mathcal H : \  \left\langle v ,  \xi   \right\rangle   \leq 0  \  \mbox{for all }   \  \xi \in T            \right\}$.
Then, for all $v\in \mathcal H$ there exists a unique decomposition
\begin{align*}
& v= v_T + v_N, \quad v_T\in T, \ v_N \in N; \\
& \left\langle v_T ,  v_N   \right\rangle  = 0.
\end{align*}
Moreover, $v_T= \mbox{{\rm proj}}_T (v)$,  and $v_N= \mbox{{ \rm proj}}_N (v)$.
\end{theorem}

\begin{proof}[Proof of Proposition \ref {P:MultiSteepestDescent}] We introduce  the notation $\C(u) = \mbox{Conv}\left\{\partial f_i(u); \ i=1,...,q \right\}$. By definition,  $-s(u)$ is the projection of the origin onto the closed convex set $N_K(u) + \C(u)$.
 Hence, using that $0\in N_K(u)$, we have for any $p \in \mathcal{C} (u)$ 
\begin{equation}\label{mult-opt4} 
	\left\langle  0 -( -s(u)), p - ( -s(u)) \right\rangle \leq 0,
\end{equation}
that is
\begin{equation}\label{mult-opt5} 
	\|s(u) \|^2  + 	\left\langle  s(u), p \right\rangle \leq 0
\end{equation}
which is the desired inequality. Verify now that $s(u)$ is a multiobjective descent direction at $u\in K \setminus \mathcal{P}_c$. By definition of $s(u)$, we can write
$s(u)= \bigg(z -N_K(u)\bigg)^{0}$ for some $z\in -\C(u)$.
Since $T_K(u)$ is the polar cone of $N_K(u)$, then by Moreau decomposition theorem,
\begin{align*}
\bigg(z -N_K(u)\bigg)^{0}&= z- \mbox{proj}_{N_K(u)}z \\
&=  \mbox{proj}_{T_K(u)}z
\end{align*} 
which shows that $s(u) \in T_K(u)$, and concludes the proof.
\end{proof}

Now that we have established that $s(u)$ is a multiobjective descent direction, one may wonder why it is called the \textit{steepest} descent direction. Observe first that in the case of a single differentiable objective function  $f$, and a constraint $K$, the direction $s(u)$ at $u\in K$ is given by
\begin{align}
s(u)&= (-N_K(u) - \nabla f(u)))^{0},\\
	&= \mbox{proj}_{T_K(u)}( - \nabla f(u)).
\end{align}
It is known that the normalized vector $\frac{s(u)}{\|s(u)\|}$ is the solution, when $\nabla f(u) \neq 0$, of the minimization problem
\[
	\min \left\{ \  df(u,d) :     \quad d\in  T_K(u), \  \|d\|=1                        \right\},
\]
whence the name of \textit{steepest descent direction} for $\mbox{proj}_{T_K(u)}( - \nabla f(u))$. As shown below, this steepest descent property can be extended to the multiobjective case:
\begin{equation}\label{E:Charac00}
\ \frac{s(u)}{\|s(u)\|} =   \mbox{\rm argmin} \left\{   \max\limits_{i=1,...,q} df_i(u,d) : \ d\in  T_K(u), \  \|d\|=1      \right\}.
\end{equation}
\noindent Moreover, in the case of a single differentiable objective function, it can be easily verified that 
\begin{equation}\label{E:Charac2}
	s(u)= \mbox{\rm argmin} \left\{ \frac{1}{2} \| v \|^2 +  df(u,d) :     \quad  d\in  T_K(u)                      \right\},
\end{equation}
and this further characterization will also  be generalized to the multiobjective case.
In addition, we make a link between the formulations (\ref{E:Charac2}) and (\ref{E:Charac00}), by introducing a continuum of characterizations  based on 
the use of
$\| \cdot \|^r$, these two situations corresponding  to $r=2$, and  the limiting case $r= +\infty$.
%
\begin{theorem}\label{steep;desc;dir} Let $u\in K \setminus \mathcal{P}_c$.    Then $s(u)$ can be formulated in the following equivalent forms: 
\begin{align}
1.& \ s(u)= \bigg(-N_K(u) - \mbox{{\rm Conv}}\left\{\partial f_i(u) \right\} \bigg)^{0}\label{D:FirstForm} \\
2.& \ \frac{s(u)}{\|s(u)\|^{\frac{r-2}{r-1}}}= \underset{d\in T_K(u)}{\mbox{{\rm argmin}}}  \left\{ \frac{1}{r} \| d \|^r     + \max\limits_{i=1,...,q} df_i(u,d) \right\} \ \text{for all } \ r \in ]1,+\infty[ \label{D:SecondForm}\\
3.& \ \frac{s(u)}{\|s(u)\|} =   \underset{\underset{\Vert v \Vert =1}{v\in T_K(u)}}{\mbox{\rm argmin}} \left\{   \max\limits_{i=1,...,q} df_i(u,d) \right\} . \label{ThirdForm}
\end{align}
\end{theorem}

\begin{remark}\label{R:DualForms} {\rm The equivalence between formulations 1. and 3. of the steepest descent direction has been first obtained, in the finite dimensional and smooth case, by 
 Cornet in \cite[Proposition 3.1]{Corn1}. Formulation 2 appears in   Fliege and Svaiter \cite{FS} in the smooth case, for $r=2$ : 
\begin{equation*}
s(u) = \underset{d\in T_K(u)}{\mbox{{\rm argmin}}}  \left\{ \frac{1}{2} \| d \|^2     + \max\limits_{i=1,...,q} df_i(u,d) \right\},
\end{equation*}
but the equivalence between formulations 1. and 2. is seemingly new.
The second formulation for $r\neq 2$ is new, although it was stressed in \cite{FS} that $\frac{1}{2} \| d \|^2$ could be replaced by any positive proper l.s.c strictly convex function which is dominated by the norm around the origin. The interest of considering $r$ arbitrary large is that we can see -at least formally- the third formulation as the limit of the second when $r \to +\infty$ : $\frac{r-2}{r-1}$ tends to $1$, while the function $ \frac{1}{r} \| \cdot \|^r$ is pointwise converging to the indicator function of the unit ball $\delta_{\mathbb B (0,1) }(\cdot)$.
}
\end{remark}
The following proof is based on duality arguments (von Neumann's min-max theorem) which were first introduced in the smooth differentiable case  in \cite{Corn1}. The extension of these results to the non-smooth case is nontrivial and requires some adjustments. In addition,  the third formulation is obtained from the second, using an epiconvergence ($\Gamma$-convergence) argument.
\begin{proof}[Proof of Theorem \ref{steep;desc;dir}] 
In all that follows, $\mathcal{C}(u)$ denotes $\mbox{{\rm Conv}}\{\partial f_i (u)\}$, where $u$ is a fixed element of $K \setminus \mathcal{P}_c$. As a consequence, $s(u) \neq 0$. 

 Let us start by proving item 2. Because of the  powered norm  term, 
 $d \mapsto \frac{1}{r} \| d \|^r     + \max\limits_{i=1,...,q} df_i(u,d)$ is a coercive strictly convex function.
  Therefore, there exists a unique solution $ \bar{d}$ to the minimization problem
 \begin{equation}\label{sv1}
 \min _{d\in  T_K(u)}  \left\{ \frac{1}{r} \| d \|^r    + \max\limits_{i=1,...,q} \max_{p_i\in  \partial f_i(u)}  \left\langle  p_i, d \right\rangle \right\}.
 \end{equation}
 Let us show that $ \bar{d}= \frac{s(u)}{\|s(u)\|^{\frac{r-2}{r-1}}}$.
We use a duality argument which relies on the equivalent formulation of (\ref{sv1}) as the convex-concave saddle value problem
 \begin{equation}\label{sv2}
 	\min_{d\in  T_K (u)}  \  \max_{p \in  \mathcal{C}(u)}  \ \left\{   \frac{1}{r} \| d \|^r  + \left\langle   p, d \right\rangle \right\}.
\end{equation}
It is
associated with the convex-concave Lagrangian function
\[
L(d, p)=  \frac{1}{r} \| d \|^r +  \left\langle p, d \right\rangle
\]
defined on $T_K (u) \times \mathcal{C}(u)$.
Since  $L$ is convex and coercive with respect to the first variable, and  $\mathcal{C}(u)$ is bounded,  by the von Neumann's minimax theorem (see \cite[Theorem 9.7.1]{ABM})
there exists $\bar p  \in \mathcal{C}(u)$ such that $ (\bar{d},\bar{p})$ is a saddle point of (\ref{sv2}), that is 
\begin{equation}\label{sd6}
  \inf_{d \in   T_K(u)}  L(d , \bar{p})=L( \bar{d}, \bar{p})=\sup_{p \in \mathcal{C}(u)}  L(\bar d, p).
\end{equation}
For any $p \in \mathcal{C}(u)$ let us define
\begin{equation}\label{sv3}
d(p):=\underset{d\in T_K(u)}{\text{argmin}} \left\{ \frac{1}{r} \Vert d \Vert^r + \langle p, d \rangle \right \}.
\end{equation}
Writing down the optimality condition for the above primal problem gives
\begin{equation}\label{sv4}
d(p)=\mbox{proj}_{T_K (u)} \left(\frac{-p}{\Vert d(p) \Vert^{r-2}}\right),
\end{equation}
which,  by Moreau's theorem, can be rewritten  as
\begin{equation}\label{sdd4}
d(p)=\frac{1}{\Vert d(p) \Vert^{r-2}} \left(-p -N_K(u)\right)^0.
\end{equation}

Observe that $(\bar d, \bar p)$ being a saddle point of $L$ implies $\bar d = d(\bar p)$. Thus we just need to prove that $d(\bar p)= \frac{s(u)}{\|s(u)\|^{\frac{r-2}{r-1}}}$.
 To identify $\bar{p}$, we use the dual formulation
\begin{equation}
\bar{p} = \underset{p \in \mathcal{C}(u)}{\mbox{argmax}} \min\limits_{d\in T_K (u)}\left\{  \frac{1}{r} \Vert d \Vert^r +\langle p,d\rangle \right \},
\end{equation}
which,  by (\ref{sv3}) and (\ref{sv4}), can be rewritten as
\begin{equation}
\bar p=\underset{p \in \mathcal{C}(u)}{\mbox{argmax}} \  \frac{1}{r} \Vert d(p) \Vert^r  -\Vert d(p) \Vert^{r-2} \langle \frac{-p}{\Vert d(p) \Vert^{r-2}}, \mbox{proj}_{T_K(u)} \left( \frac{-p}{\Vert d(p) \Vert^{r-2}} \right) \rangle.
\end{equation}
Using  Moreau's theorem, we obtain
\begin{equation}
\bar p=\underset{p \in \mathcal{C}(u)}{\mbox{argmax}} \  \frac{1}{r} \Vert d(p) \Vert^r  -\Vert d(p) \Vert^{r-2} \Vert \mbox{proj}_{T_K(u)} \left( \frac{-p}{\Vert d(p) \Vert^{r-2}} \right) \Vert^2,
\end{equation}
which, by (\ref{sv4}) and $r \in ]1,+\infty]$, is  equivalent to
\begin{equation}
\bar p=\underset{p \in \mathcal{C}(u)}{\mbox{argmin}} \ \Vert d(p) \Vert^{r-1}.
\end{equation}
From (\ref{sdd4},) we know that $\Vert d(p) \Vert^{r-1}= \Vert \left(-p -N_K(u)\right)^0 \Vert$. Therefore, $s(u)= \left(-\bar p -N_K(u)\right)^0$ with $\Vert d(\bar p) \Vert ^{r-1} = \Vert s(u) \Vert$. Using again (\ref{sdd4}), we obtain $ d(\bar{p})= \frac{s(u)}{\|s(u)\|^{\frac{r-2}{r-1}}}$, as expected.

Let us complete the proof by proving the third characterisation. As we said in Remark \ref{R:DualForms}, it relies on a limit argument. Define, for any $r>1$, the functions $F_r : d \in \H \mapsto \frac{1}{r}\Vert d \Vert^r + \max\limits_{i=1,...,q} df_i(u,d) + \delta_{T_K(u)}(d)$. It can be easily verified that the sequence $(F_r)_{r>1}$ epiconverges when $r\to +\infty$ to $$F : d \in \H \mapsto \delta_{\{\Vert \cdot \Vert \leq 1\}} (d) + \max\limits_{i=1,...,q} df_i(u,d) + \delta_{T_K(u)}(d).$$
From (\ref{D:SecondForm}) and \cite[Theorem 12.1.1]{ABM} we can deduce that
\begin{equation}\label{sdd5}
\frac{s(u)}{\|s(u)\|} =   \underset{\underset{\Vert v \Vert \leq 1}{v\in T_K(u)}}{\mbox{\rm argmin}} \left\{   \max\limits_{i=1,...,q} df_i(u,d) \right\},
\end{equation}
where the inequality constraint $\Vert v \Vert \leq 1$ can be replaced by $\Vert v \Vert = 1$, since $\frac{s(u)}{\|s(u)\|}$ is a normalized vector.

\end{proof}

}

\subsection{The Multi-Objective Gradient dynamic}

{

Here we present and discuss the continuous dynamic governed by the multiobjective steepest descent vector field $u \mapsto s(u)$.

In \cite{Sm1}, Smale defined the notion of \textit{gradient process} for the multiobjective optimization problem (CMO). It is a differential equation 
\begin{equation} \label{mult-opt2}
\dot{u}(t) = \phi (u(t))
\end{equation}
where $\phi: K \rightarrow \mathcal H$  is a mapping which satisfies the following properties:
\begin{equation}
  \label{eq:m-gradient}
  \left\{\begin{array}{l}
  \phi(u) \ \text{\rm is a multiobjective descent direction whenever } u \notin \mathcal{P}_c, \\
  \phi(u)=0 \ \mbox{ if}  \ u\in {\mathcal P}_c  \ .
  \end{array}\right.
\end{equation}
The interest of such a gradient process is twofold : the stationary points of the dynamic are exactly the critical Pareto points, and as long as $u(t)$ is not a critical Pareto point, all the objective functions are decreasing. Clearly, from its definition and Proposition \ref{P:MultiSteepestDescent}, the vector field $u \mapsto s(u)$ induces a gradient process, defined as follows :

\begin{definition}\label{D:MOGS}
The dynamical system which is governed by the vector field  $u \mapsto s(u)$, is called the  Multi-Objective Gradient system. Its solution trajectories
$t\mapsto u(t)$ verify
\begin{equation}\label{E:MOGS}
\mbox{{\rm (MOG)}} \quad   \dot u(t) + \bigg(N_K(u(t)) + \mbox{{\rm Conv}}\left\{\partial f_i(u(t)) \right\} \bigg)^{0}= 0.
\end{equation} 
\end{definition}

\begin{remark}
{\em 
Instead of considering the vector field $u\mapsto s(u)$ to govern our dynamic, we could have chosen one of the directions that appear in Theorem \ref{steep;desc;dir}. In fact, each of them induces a gradient process. From the viewpoint of the dynamic system, these directions generate the same integral curves, with a different time scale.
}
\end{remark}

In \cite{Sm1},  the vector field $\phi$ governing the gradient process is continuous, in a finite dimensional setting. In our context, the corresponding notions have been  extended  in order to cover dynamical systems governed by a discontinuous vector field on a general Hilbert space, as the (MOG) dynamic. In particular, instead of classical (continuously differentiable) solutions, we will consider strong solutions (absolutely continuous on bounded time intervals), the equality  (\ref{E:MOGS}) being satisfied almost everywhere. Let us make precise this (see \cite [Appendix]{Br}  for more details):

 \begin{definition} 
 \label{abs.cont.def} Given $T\in\mathbb{R}^{+}$, a function
$u:\left[0,T\right]\rightarrow \mathcal H$ is said to be absolutely continuous
if one of the following equivalent properties holds:

 $i)$ there exists an integrable function $g:\left[0,T\right]\rightarrow \mathcal H$
such that $u\left(t\right)=u\left(0\right)+\int_{0}^{t}g\left(s\right)ds \ \ \forall t\in\left[0,T\right];$

$ii)$ $u$ is continuous and its distributional derivative
belongs to the Lebesgue space $L^{1}\left(\left[0,T\right];\mathcal H\right)$; 

 $iii)$ for every $\epsilon>0$, there exists  $\eta>0$
such that for any finite family of intervals $I_{k}=\left(a_{k},b_{k}\right)$,  \
$I_{k}\cap I_{j}=\emptyset$ for $k\neq j$ and $\sum_{k} |b_{k}-a_{k}|\leq\eta\Longrightarrow\sum_{k} \|u\left(b_{k}\right)-u\left(a_{k}\right)\| \leq\epsilon.$
\end{definition}

\noindent  Moreover, an absolutely continuous function is
differentiable almost everywhere, its derivative coincide with its
distributional derivative almost everywhere, and one can recover the
function from its derivative $u^{\prime}=g$  using the integration formula 
$\left(i\right).$
We can now make precise the notion of solution for the (MOG) dynamic (recall that $\S^q$ denotes the unit simplex in $\R^q$).
 \begin{definition} \label{def.sol-basic} We say that  $u(\cdot)$ is
a strong global solution of (MOG)
if the following properties are satisfied:

 $\left(i\right)$ $u: [0,+\infty [ \rightarrow \mathcal H$
is  absolutely continuous on each  interval
$\left[0,T\right]$, $0<T<+\infty;$

 $\left(ii\right)$ there exists $\eta: [0,+\infty [ \rightarrow \mathcal H$,  $v_i: [0,+\infty [ \rightarrow \mathcal H$,  $\theta_i: [0,+\infty [ \rightarrow [0,1]$ i=1,2,...,q  
    \ which satisfy
 \begin{align}
 &  \theta_i \in L^{\infty}(0,+\infty; \R), \ \  \left(\theta_i(t)\right)\in \S^q \text{ for almost all } t>0;\\
 & v_i\in L^{\infty}(0,T; \mathcal H), \quad  \eta \in L^2(0,T; \mathcal H) \  \mbox{for all} \ T>0   \  \mbox{ and all} \  i=1,2,..., q; \\
 & \eta(t) \in N_K(u(t)), \  v_i(t) \in \partial f_i(u(t)) \quad \mbox{for almost all} \ t>0; \label{def1}\\
 &  \dot u(t) +  \eta (t) + \sum_i \theta_i (t) v_i(t)   = 0 \quad \quad \  \mbox{for almost all} \ t>0; \label{def2}\\
 &  \dot u(t) + ( N_K(u(t)) +    \mbox{{\rm Conv}}\left\{ \partial f_i(u(t)) \right\} )^{0}     = 0 \quad \mbox{for almost all} \ t>0.
 \end{align}
\end{definition} 
Now we can establish the first qualitative properties of strong solutions of (MOG). First, we show that trajectories satisfy a local Lipschitz continuity property. Second, as announced, we show that the objective functions are decreasing along the trajectories. 
\begin{proposition}\label{P:qualitative}
Let us make assumptions {\rm H0), H1)}. Then for any  strong global solution $t\in [0, +\infty [ \mapsto u(t) \in \mathcal H$ of {\rm(MOG)}, the following properties hold: 

\noindent 	$i)$ \textit{Descent property}: \ for each $i=1,...,q$,  \   $t \mapsto f_i(u(t))$ is  a nonincreasing absolutely continuous function,  and for almost all $t>0$
\begin{equation}\label{E:DecreaseProperty} 
	 \frac{d}{dt}  f_i(u(t)) \leq - \|\dot{u}(t)\|^2 .
\end{equation}

\noindent $ii)$ \textit{Lipschitz continuity}: \ The trajectory $u$ is Lipschitz continuous on any finite time interval $[0,T]$. If moreover it is bounded, $u$ is Lipschitz continuous on  $[0,+\infty[$.
\end{proposition}

\noindent This Proposition is a direct consequence of Proposition \ref{P:MultiSteepestDescent} and the following generalized chain rule from Br\'ezis:
\begin{lemma}{\rm {\cite[Lemma 4, p.73]{Br}}} \label{lemma_Brezis} Let  $\Phi: \mathcal H \rightarrow \R \cup \{+\infty\}$ be a closed convex proper function. Let
  $u\in L^2(0, T; \mathcal H)$ be such that $\dot{u}\in L^2 (0, T; \mathcal H)$, and
  $u(t)\in {\rm \mbox{dom}}(\partial \Phi)\ \mbox{for a.e.} \ t$. Assume that there exists
  $\xi\in L^2(0, T; \mathcal H)$ such that $\xi(t)\in \partial \Phi(u(t))$ for
  a.e. $t$. Then the function $t\mapsto \Phi(u(t))$ is absolutely
  continuous, and for every $t$ such that $u$ and  $\Phi(u)$ are differentiable at $t$, and $u(t)\in {\rm \mbox{dom}}(\partial \Phi)$, we have
$$
\forall p\in  \partial \Phi(u(t)),\qquad \frac{d}{dt}  \Phi(u(t))=\langle \dot{u}(t), p\rangle.
$$
\end{lemma}

\begin{proof}[Proof of Proposition \ref{P:qualitative}]
$i)$ By definition of (MOG), for almost all \ $t>0$,  $\dot{u}(t) = s(u(t))$ holds.
Hence, using Proposition \ref{P:MultiSteepestDescent}, for any $p \in \mbox{Conv}\left\{\partial f_i(u(t)) \right\}$
\begin{equation}\label{qua1} 
	\Vert \dot u(t) \Vert^2 + \left\langle  \dot{u}(t), p \right\rangle \leq 0.
\end{equation}
Moreover, for almost all \ $t>0$, there exists ${v}_i (t)  \in \partial f_i(u(t))$ with  \  $v_i\in L^2(0,T; \mathcal H)$. Hence taking $p = {v}_i (t) \in \mbox{Conv}\left\{\partial f_i(u(t)) \right\}$ in (\ref{qua1}) yields
\begin{equation}\label{qua2} 
	\Vert \dot u(t) \Vert^2 + \left\langle  \dot{u}(t), v_i(t) \right\rangle \leq 0.
\end{equation}
The  derivation chain rule is valid in our situation, see Lemma \ref{lemma_Brezis}. Hence, $f_i(u)$ is absolutely continuous on each bounded interval $[0,T]$,
which, by (\ref{qua2}), gives      for almost all $t>0$
\begin{equation}\label{qua3} 
	\|\dot{u}(t)\|^2  +  \frac{d}{dt} f_i(u(t))  \leq 0.
\end{equation}
As a consequence, $\frac{d}{dt} f_i(u(t))  \leq 0$, and for each $i=1,...,q$ \ the function  $t \mapsto f_i(u(t))$ is  nonincreasing. 

\noindent $ii)$ By Definition \ref{def.sol-basic} of a  strong global solution, we can write  
$$
\dot u(t) +  \eta (t) + \sum_i \theta_i (t) v_i(t)  = 0
$$
with $\eta(t) \in N_K(u(t)), \  v_i(t) \in \partial f_i(u(t)), \left(\theta_i(t)\right)\in \S^q  \quad \mbox{for almost all} \ t>0$.
Let us argue on some $[0,T]$. Since $\partial f_i$ is bounded on bounded sets, and $\left(\theta_i(t)\right)\in \S^q$, we have
$$
\dot u(t) +  \eta (t)  = g(t)
$$
with $g:= - \sum_i \theta_i  v_i \in   L^{\infty} (0,T; \mathcal H)$. Taking the scalar product with $\dot u(t)$, we obtain 
\begin{equation}\label{quali1}
\| \dot u(t)\|^2 + \langle \eta (t) , \dot u (t) \rangle = \left\langle g(t), \dot u(t)  \right\rangle.
\end{equation}
Note that the normal cone mapping $N_K$ is the subdifferential of $\delta_K$, the indicator function of $K$. Using the derivation chain rule, we have
\begin{equation}\label{quali2} 
 \left\langle  \eta(t),  \dot u(t)  \right\rangle= \frac{d}{dt} \delta_K (u(t))  =0.
\end{equation}
Combining (\ref{quali1}), (\ref{quali2}) and the Cauchy-Schwarz inequality, we obtain 
$$
\| \dot u(t)\|  \leq \| g(t)\|.
$$
Hence, $\dot u \in L^{\infty} (0,T; \mathcal H)$, which is equivalent to the Lipschitz continuity of $u$. If $u$ is bounded, just notice that $g \in L^{\infty} (0,+\infty; \mathcal H)$, and conclude in a similar way.
\end{proof}

 \subsection{Examples}\label{S:Examples}
We now illustrate the (MOG) dynamic through some simple examples in $\mathcal H = \mathbb R \times \mathbb R$. They suggest that its study is nontrivial, because (MOG) is governed by a vector field that  is neither monotone, nor locally Lipschitz: without any further assumption, we cannot expect more than the  H\"older continuity  of this field vector (see \cite{AG1}, and Example 3 below).

\begin{ejem}
{\em 
Take the quadratic functions $f_1(x,y)= \frac{1}{2}(x+1)^2    + \frac{1}{2}y^2 $  and  $f_2(x,y)= \frac{1}{2}(x-1)^2    + \frac{1}{2}y^2 $. The corresponding Pareto set is $\P=\P_w=[-1, +1] \times \left\{0\right\}$ and the steepest descent is given by  :

\begin{equation}
s(x,y)=
\begin{cases}
 -(x-1,y)     \  \ \mbox{if } \  x>1,\\
 -(0,y) \  \quad \quad \mbox{if } \  -1 \leq x \leq 1,\\
 -(x+1,y)     \ \  \mbox{if } \   x<-1 .
\end{cases}
\end{equation}

\noindent Figure 1 shows some trajectories of the (MOG) dynamic. Trajectories are straight lines connecting the starting point and its projection on $\P$. 
On this  example, we can observe that (MOG) is different from the descent dynamics associated with a  scalarized function $\alpha f_1 + f_2$, $\alpha >0$. It is neither related to the descent dynamic associated to the max function $f= \max f_i$. Indeed, when starting from some $(x_0,y_0)$ with $x_0 >1$,  the trajectory of the steepest descent for $f= \max f_i$ is first oriented 
toward  the Pareto equilibrium $(-1,0)$, while the trajectory of (MOG) is oriented toward $(1,0)$.

\vspace{0.5cm}
\setlength{\unitlength}{7cm}
\begin{picture}(0.3,0.7)(-0.6,-0.08)

\put(1.25,0.02){$x$}
\put(0.38,0.56){$y$}
\put(0.37,0.02){$0$}
\put(0.75,0.02){$+1$}
\put(0.02,0.01){$-1$}

\put(0.7,-0.1){\line(0,1){0.7}}
\put(0.0,-0.1){\line(0,1){0.7}}

\put(0.35,-0.1){\vector(0,1){0.7}}	
\put(-0.3,0){\vector(1,0){1.6}}

\put(0.06,0.00){\line(1,0){0.58}}  
\put(0.0,0.0){\line(1,0){0.7}}

\linethickness{0.3mm}
\put(0.923,0.45){\vector(-1,-2){0.03}}
\put(-0.198,0.4){\vector(1,-2){0.03}}
\put(0.5,0.4){\vector(0,-1){0.05}}

\linethickness{0.4mm}
\put(0.7,0.00){\line(1,2){0.25}}
\put(0.0,0.00){\line(-1,2){0.25}} 
\put(0.5,0.00){\line(0,1){0.5}}

\end{picture}

\begin{center}
Figure 1
\end{center}
\vspace{0.5cm}
}
\end{ejem}

\begin{ejem}
{\em 
Let $f_1(x,y)= \frac{1}{2}x^2$  and $f_2(x,y)= \frac{1}{2}y^2$. Here $\P = \{(0,0)\}$ and $\P_w = \mathbb R \times \left\{0\right\} \cup   \left\{0\right\} \times \mathbb R$. The multiobjective steepest descent vector field, once computed, is:

\begin{equation*}
s(x,y) = -\left(  \frac {xy^2}{x^2 + y^2}, \   \frac {yx^2}{x^2 + y^2}        \right) \text{ if } (x,y)\neq (0,0), \ s(0,0)= (0,0).
\end{equation*} 

\noindent Observe that $(x,y) \mapsto s(x,y)$ is a nonlinear vector field, and it is not a gradient vector field. Moreover, trajectories tend to move away from each other (see Figure 2), which reflects the fact that $(x,y) \mapsto - s(x,y)$ is not a monotone operator. Indeed, for $x>0, \ y>0, \  x\neq y$ 
$$
\left\langle  -s(x,y) +  s(y,x), (x,y) - (y,x)\right\rangle   =  -2\frac{xy(x-y)^2}{x^2 + y^2} <0.
$$
\if{Nevertheless,  in this example, it can be shown  that $(x,y) \mapsto - s(x,y)$ is \textit{hypomonotone} on $(\R^*)^2$, i.e.,
 locally,  there exists some $\alpha \geq 0$ such that $-s + \alpha I$ is monotone. To see this, compute the Jacobian of $-s$ 
 
\begin{equation*}
D_{(x,y)} (-s) = \frac{1}{(x^2 + y^2)^2} \begin{pmatrix}
y^4 - x^2y^2 & 2xy^3 \\ 2x^3y & x^4 -x^2y^2
\end{pmatrix},
\end{equation*}
and observe that $D_{(x,y)} (-s) + I$ is positive whenever $(x,y) \neq (0,0)$.}\fi

\vspace{0.1cm}
\setlength{\unitlength}{7cm}
\begin{picture}(0.5,0.7)(-0.6,0)
\put(-0.04,0.01){$0$}

\put(0.785,-0.04){$x$}
\put(-0.05,0.385){$y$}
\put(0.52,0.2){$s(x,y)$}
\put(0.0,-0.1){\vector(0,1){0.8}}	
\put(-0.1,0){\vector(1,0){1.01}}

\put(0.65,0.6){$x=y$}	
\put(0.8,0.41){$M  (x,y)$}
\put(0.8,0.4){\line(0,-1){0.4}}
\put(0.8,0.4){\line(-1,0){0.8}}	
\put(0,0.4){\line(2,-1){0.8}}	
\put(0.8,0.4){\vector(-1,-2){0.16}}

\linethickness{0.4mm}
\put(0,0){\line(1,1){0.65}}
\qbezier(0.69,0)(0.71,0.155)(0.8,0.4)

\qbezier(0.45,0.43)(0.2,0.15)(0.19,0.135)
\qbezier(0.19,0.135)(0.18,0.12)(0.17,0.1065)
\qbezier(0.17,0.1065)(0.16,0.09)(0.15,0.07)
\qbezier(0.15,0.07)(0.14,0.045)(0.135,0.025)
\qbezier(0.135,0.025)(0.133,0.0115)(0.1326,0)

\qbezier(0.0,0.415)(0.17,0.45)(0.5,0.65)
\qbezier(0.0,0.345)(0.17,0.385)(0.55,0.65)
\qbezier(0.0,0.25)(0.17,0.3)(0.6,0.65)
\end{picture}
\begin{center}
Figure 2
\end{center}
\vspace{0.5cm}
}
\end{ejem}

\begin{ejem}
{\em

Let $f_1 (x,y) =\frac{1}{2}(x^2 + y^2)$ and   $f_2 (x,y)= x$. The corresponding Pareto set is $\P=\P_w=]-\infty, 0] \times \{0\}$. Once computed, we see that the steepest descent vector field is defined according to three areas of the plane (see Figure 3):  
\begin{equation}
s(x,y)=
\begin{cases}
 -(1,0)   \quad   \quad   \quad   \quad   \quad \quad  \ \  \quad \mbox{if } \  x\geq 1,\\
 -(x,y) \    \quad   \quad   \quad   \quad    \quad   \  \quad  \quad \mbox{if } \  (x-\frac{1}{2})^2 + y^2 \leq \frac{1}{4},\\
 \frac{-1}{(x-1)^2+y^2} (y^2,y(1-x))      \ \  \  \mbox{else} .
\end{cases}
\end{equation}

\noindent As in the previous example, this vector field is neither linear nor a gradient vector field, or monotonous. Moreover, it is not locally Lipschitz. The lack of Lipschitz continuity occurs at the point $(1,0)$, where the vector field "splits" into three parts. Figure 4 provides a simple example of parameterized vectors $u_\theta,v_\theta$  that converge both to $(1,0)$ when $\theta$ goes to zero, but such that $\Vert s(u_\theta) - s(v_\theta)\Vert= \sin(\theta)$ and $\Vert u_\theta - v_\theta \Vert = \sin(\theta)\tan(\theta)$. As a consequence, $\frac{\Vert s(u_\theta) - s(v_\theta)\Vert}{\Vert u_\theta - v_\theta \Vert} \simeq \frac{\theta}{\theta^2}$ is unbounded when $\theta \to 0$.

\vspace{3.5cm}
\setlength{\unitlength}{3cm}

\begin{picture}(0.5,0.5)(-2,0)

\put(-0.08,-0.1){$0$} 

\put(0.0,-0.5){\vector(0,1){2}} 	
\put(-1,0.0){\vector(1,0){2.5}} 

\put(1,-0.5){\line(0,2){2}} 
\put(1.05,-0.5){$x=1$}

\put(0.5,0){\circle{1}}
\put(1.01,0.05){$(1,0)$}
\put(1,0){\circle*{0.04}}

\linethickness{0.4mm}

\qbezier(1.5,1.4)(1.25,1.4)(1,1.4)
\qbezier(1,1.4)(0.9,1.398)(0.8,1.39)
\qbezier(0.8,1.39)(0.7,1.37)(0.61,1.35)
\qbezier(0.61,1.35)(0.51,1.32)(0.43,1.29)
\qbezier(0.43,1.29)(0.34,1.25)(0.23,1.18)
\qbezier(0.23,1.18)(0.1,1.09)(0,1)
\qbezier(0,1)(-0.1,0.89)(-0.17,0.8)
\qbezier(-0.17,0.8)(-0.23,0.7)(-0.28,0.6)
\qbezier(-0.28,0.6)(-0.33,0.51)(-0.36,0.4)
\qbezier(-0.36,0.4)(-0.39,0.3)(-0.41,0.2)
\qbezier(-0.41,0.2)(-0.42,0.1)(-0.425,0)

\qbezier(1.5,1)(1.25,1)(1,1)
\qbezier(1,1)(0.9,0.99)(0.8,0.985)
\qbezier(0.8,0.985)(0.7,0.963)(0.61,0.93)
\qbezier(0.61,0.93)(0.53,0.897)(0.42,0.83)
\qbezier(0.42,0.83)(0.32,0.76)(0.22,0.66)
\qbezier(0.22,0.66)(0.148,0.56)(0.065,0.415)
\qbezier(0.065,0.415)(0.04,0.36)(0,0.22)
\qbezier(0,0.22)(-0.01,0.175)(-0.025,0)

\qbezier(0,1.4)(-0.09,1.328)(-0.2,1.22)
\qbezier(-0.2,1.22)(-0.3,1.13)(-0.345,1.07)
\qbezier(-0.345,1.07)(-0.4,1)(-0.5,0.847)
\qbezier(-0.5,0.847)(-0.56,0.74)(-0.6,0.65)
\qbezier(-0.6,0.65)(-0.639,0.55)(-0.669,0.45)
\qbezier(-0.669,0.45)(-0.69,0.35)(-0.71,0.25)
\qbezier(-0.71,0.25)(-0.72,0.15)(-0.73,0)

\qbezier(1.5,0.6)(1.25,0.6)(1,0.6)
\qbezier(1,0.6)(0.9,0.592)(0.8,0.57)
\qbezier(0.8,0.57)(0.7,0.525)(0.64,0.48)
\qbezier(0.64,0.48)(0.32,0.24)(0,0)

\qbezier(1.5,0.4)(1.25,0.4)(1,0.4)
\qbezier(1,0.4)(0.95,0.397)(0.9,0.389)
\qbezier(0.9,0.389)(0.85,0.375)(0.84,0.368)
\qbezier(0.84,0.368)(0.42,0.184)(0,0)

\qbezier(1.5,-0.2)(1.25,-0.2)(1,-0.2)
\qbezier(1,-0.2)(0.98,-0.198)(0.96,-0.196)
\qbezier(0.96,-0.196)(0.48,-0.098)(0,0)

\end{picture}
\vspace{1.5cm}

\begin{center}
Figure 3
\end{center}

\vspace{0.5cm}
\setlength{\unitlength}{6cm}

\begin{picture}(0.6,0.6)(-0.75,0)
\put(-0.15,0){$(0,0)$}
\put(0,0){\circle*{0.02}}

\put(0.138,0){\line(-1,3){0.02}}

\put(0.145,0.010){$\theta$}
\put(1.01,0){$(1,0)$}
\put(1,0){\circle*{0.02}}

\put(0.933,0.14){$\theta$}


\put(1,0.133){\line(-3,-1){0.06}}

\put(0.95,0){\line(0,1){0.05}}
\put(1,0.05){\line(-1,0){0.05}}
\put(1.005,0.51){$v_\theta$}
\put(0.76,0.43){$u_\theta$}
\put(1,0.5){\circle*{0.02}}
\put(0.8,0.4){\circle*{0.02}}

\put(0.737,0.37){\line(1,-2){0.03}}

\put(0.764,0.31){\line(2,1){0.065}} 
\put(0,0){\line(2,1){1}}
\put(0,0){\line(1,0){1}}
\put(0.5,0){\oval[0.5](1,1)[t]}

\put(0.95,0){\line(0,1){0.05}}
\put(1,0.05){\line(-1,0){0.05}}
\put(1,0){\line(-1,2){0.2}}	
\put(1,0){\line(0,1){0.6}}
\end{picture}

\begin{center}
Figure 4
\end{center}

\vspace{0.5cm}

}
\end{ejem}

\subsection{Related dynamics}\label{S:RelatedDynamics}

When there is just one objective function $f$, since  $\partial f(u(t))$ is a closed convex set, the (MOG) system specializes to 
\begin{equation*}
 \dot u(t) + \bigg(N_K(u(t)) + \partial f(u(t)) \bigg)^{0}= 0.
\end{equation*}
Indeed, this system is equivalent to
\begin{equation*}
 \dot u(t) + N_K(u(t)) + \partial f(u(t))\ni 0,
\end{equation*}
because, in this case, the lazy solution property is automatically satisfied by the trajectories of the semigroup of contractions generated by the maximal monotone operator
$N_K + \partial f$, see \cite[Theorem 3.1]{Br}. In particular, our existence and asymptotic analysis for (MOG) in Sections 2 and 3 extends the well-known  results for the nonsmooth gradient flow (see \cite{Br}). 

This leads us to ask a natural question, which is the study of the relationship (or differences) between (MOG) and the Multiobjective Differential Inclusion ((MDI) for short)
\begin{equation}\label{mult-steep-desc22}
{\rm(MDI)} \quad \dot u(t) + N_K(u(t)) + \mbox{Conv}\left\{\partial f_i(u(t))\right\} \ni 0.\\
\end{equation}
It appears that (MDI) enjoys a weaker form of Proposition \ref{P:qualitative} $i)$ :
\begin{proposition} \label{diffinclusion}  Let  $t\in [0, +\infty [ \mapsto u(t) \in \mathcal H$ be a strong global solution of {\rm(MDI)} in the sense of Definition \ref{def.sol-basic} (except the lazy property). Then for almost all $t\geq 0$, such that         
 $u(t)\notin {\mathcal P}_c$, there exists some $i\in \left\{1,\cdots,q\right\}$ (which depends on $t$) such that
 $$
 \frac{d}{dt} f_i(u(t))  < 0.
 $$
\end{proposition}
\begin{proof}
Since $u$ is a strong solution of (MDI), there exists $\eta: [0,+\infty [ \rightarrow \mathcal H$,  $v_i: [0,+\infty [ \rightarrow \mathcal H$,  $\theta_i: [0,+\infty [ \rightarrow [0,1]$ $i=1,2,...,q$  
 which satisfy
 \begin{align*}
 &  \theta_i \in L^{\infty}(0,+\infty; \R), \ \  \left(\theta_i(t)\right)\in \S^q \text{ for almost all } t>0;\\
 & v_i\in L^{\infty}(0,T; \mathcal H), \quad  \eta \in L^2(0,T; \mathcal H) \  \mbox{for all} \ T>0   \  \mbox{ and all} \  i=1,2,..., q; \\
 & \eta(t) \in N_K(u(t)), \  v_i(t) \in \partial f_i(u(t)) \quad \mbox{for almost all} \ t>0; \\
 &  \dot u(t) +  \eta (t) + \sum_i \theta_i (t) v_i(t)   = 0 \quad \quad \  \mbox{for almost all} \ t>0;
 \end{align*}
Taking the scalar product of the above equation with $\dot u(t)$, we obtain
\begin{equation}\label{diff} 
\| \dot u(t)	\|^2 + \left\langle \eta(t)  , \dot u(t)  \right\rangle + \sum_i \theta_i (t)\left\langle  v_i(t),  \dot u(t)  \right\rangle  = 0.
\end{equation}
The  derivation chain rule is valid in our situation, see Lemma \ref{lemma_Brezis}. Hence, $f_i(u)$ is absolutely continuous on each bounded interval $[0,T]$,
which gives,      for almost all $t>0$
\begin{equation}\label{diff10} 
 \frac{d}{dt} f_i(u(t))  = \left\langle  v_i(t),  \dot u(t)  \right\rangle.
\end{equation}
By a similar argument  using the indicator function $\delta_K$  of $K$ (recall (\ref{quali2})) we have
\begin{equation}\label{diff11} 
 0= \frac{d}{dt} \delta_K (u(t))  = \left\langle  \eta(t),  \dot u(t)  \right\rangle.
\end{equation}
Combining (\ref{diff})  with (\ref{diff10}), (\ref{diff11}), we obtain 
\begin{equation}\label{diff12} 
\| \dot u(t)	\|^2  + \sum_i \theta_i (t) \frac{d}{dt} f_i(u(t))  \leq 0.
\end{equation}
Since  $u(t)\notin {\mathcal P}_c$, we have $\dot u(t) \neq 0$. Hence
\begin{equation}\label{diff13} 
\sum_i \theta_i (t)\frac{d}{dt} f_i(u(t))  < 0.
\end{equation}
Since $\left(\theta_i(t)\right)\in \S^q$, this clearly implies that at least one of the derivatives $\frac{d}{dt} f_i(u(t))$ is negative.
\end{proof}

\begin{remark}\label{max function} {\rm 
Proposition \ref{diffinclusion} tells us that, for any trajectory of (MDI), for almost all $t>0$,  at least one of the objective functions  decreases. 
We will illustrate this on a few examples, and highlight the fact that, by contrast, for (MOG) trajectories, they are all decreasing.

a) Consider the steepest descent dynamic associated to one of the objective functions, say $f_i$, $i$ being fixed.
Clearly, $\partial f_i(u) \subset \mbox{Conv}\left\{\partial f_j(u) \right\}$, and the corresponding trajectories are solutions of (MDI). The strategy consisting in taking care  of 
only one objective function $f_i$, clearly leads to  Pareto equilibria, but fails in general to improve all  the objective functions.

b) The scalarization approach consists in taking a constant convex combination of the objective functions $f_\theta = \sum\limits_{i=1}^{q} \theta_i f_i$, with $\theta \in \S^q$. The sum rule for continuous convex functions gives, for any $u \in \mathcal H$ 
\begin{equation}
\partial f_\theta (u) = \sum\limits_{i=1}^{q} \theta_i \partial f_i (u),
\end{equation}
and clearly $\partial f_\theta (u) \subset \mbox{Conv}\left\{\partial f_i(u)\right\}$. 
By Bruck's theorem \cite{Bruck}, any orbit of the generalized gradient flow generated by $\partial f_\theta$ converges to a minimizer of $f_\theta$,
 which, by Lemma \ref{equiv},  is a weak Pareto optimal point. But, in general, this approach fails  to improve all  the objective functions.
 Take for instance in Example 1 any $\theta=(\lambda,(1-\lambda))$ for $\lambda \in ]0,1]$. When starting from $(1,0)$, the trajectory goes straight to $(1-2\lambda,0)$ by decreasing $f_1$ but increasing $f_2$.

c) Consider the steepest descent dynamic associated to the  function $f= \max_i f_i$. 
This dynamic has some similarities with (MOG), but it is different.
As a  supremum of a finite number of convex continuous  functions, $f$ is still  convex continuous. The classical subdifferential rule for the supremum of convex functions, 
see for example \cite[Theorem 18.5]{BC}, gives (in our setting)
\begin{equation}\label{diff14} 
\partial f(u) = \mbox{Conv}\left\{\partial f_i(u): \   i \in I(u)\right\}
\end{equation}
where $I(u)= \left\{  i \in I:  \ f_i(u)= f(u) \right\}$ is the set of the active indices at $u$.
Clearly $ \partial f(u) \subset  \mbox{Conv}\left\{\partial f_i(u)\right\}$. As a consequence, the  trajectories of the steepest descent for $f= \max f_i$ are also solutions of (MDI).
But, in general, they fail to satisfy that all the objective functions are decreasing.
Take for instance Example 1: 
when starting from some $(x_0,y_0)$ with $x_0 >y_0 > 1$, along the trajectory $f_2$ is first decreasing, until the current point reaches the projection of $(1,0)$ on the line segment joining $(x_0,y_0)$ to $(-1, 0)$, then it is increasing.


d) As shown by the above examples, (MDI) provides diversity, an interesting feature for evolutionary processes, and generating the whole Pareto set, see \cite{BrS}.

}
\end{remark}


\section{Asymptotic convergence to a weak Pareto minimum}\label{asp}

In this section, we study the asymptotic behavior (as $t\rightarrow +\infty$) of the strong global solutions of (MOG). We take for granted their existence, this question being examined into detail in  section \ref{ex-yo}. 
In order to prove the weak convergence of the trajectories of (MOG), we use the classical Opial's lemma \cite{Op}. We recall its statement in its continuous form, and  give a short proof of it:

\if{
Firstly, let us make precise the notion of strong solution. Indeed, we take for granted the existence of  (MOG), this question being examined into detail in  section \ref{ex-yo}.

\subsection{Strong solution}
Let us recall some classical facts concerning vector-valued
functions of real variables (see \cite [Appendix]{Br}  for more details).
 \begin{definition} 
 \label{abs.cont.def} Given $T\in\mathbb{R}^{+}$, a function
$u:\left[0,T\right]\rightarrow \mathcal H$ is said to be absolutely continuous
if one of the following equivalent properties holds:

 $i)$ there exists an integrable function $g:\left[0,T\right]\rightarrow \mathcal H$
such that $u\left(t\right)=u\left(0\right)+\int_{0}^{t}g\left(s\right)ds \ \ \forall t\in\left[0,T\right];$

$ii)$ $u$ is continuous and its distributional derivative
belongs to the Lebesgue space $L^{1}\left(\left[0,T\right];\mathcal H\right)$; 

 $iii)$ for every $\epsilon>0$, there exists  $\eta>0$
such that for any finite family of intervals $I_{k}=\left(a_{k},b_{k}\right)$,  \
$I_{k}\cap I_{j}=\emptyset$ for $k\neq j$ and $\sum_{k} |b_{k}-a_{k}|\leq\eta\Longrightarrow\sum_{k} \|u\left(b_{k}\right)-u\left(a_{k}\right)\| \leq\epsilon.$
\end{definition} 
  Moreover, an absolutely continuous function is
differentiable almost everywhere, its derivative coincide with its
distributional derivative almost everywhere, and one can recover the
function from its derivative $u^{\prime}=g$  using the integration formula 
$\left(i\right).$
We can now make precise the notion of solution for (MOG).
 \begin{definition} \label{def.sol-basic} We say that  $u(\cdot)$ is
a strong global solution of 
\begin{center}
$\mbox{{\rm (MOG)}} \quad   \dot u(t) + \bigg(N_K(u(t)) + \mbox{{\rm Conv}}\left\{\partial f_i(u(t)) \right\} \bigg)^{0}= 0$
\end{center}
if the following properties are satisfied:

 $\left(i\right)$ $u: [0,+\infty [ \rightarrow \mathcal H$
is  absolutely continuous on each  interval
$\left[0,T\right]$, $0<T<+\infty;$

 $\left(ii\right)$ there exists $\eta: [0,+\infty [ \rightarrow \mathcal H$,  $v_i: [0,+\infty [ \rightarrow \mathcal H$,  $\theta_i: [0,+\infty [ \rightarrow [0,1]$ i=1,2,...,q  
    \ which satisfy
 \begin{align}
 &  \theta_i \in L^{\infty}(0,+\infty; \R), \ \  \left(\theta_i(t)\right)\in \S^q \text{ for almost all } t>0;\\
 & v_i\in L^{\infty}(0,T; \mathcal H), \quad  \eta \in L^2(0,T; \mathcal H) \  \mbox{for all} \ T>0   \  \mbox{ and all} \  i=1,2,..., q; \\
 & \eta(t) \in N_K(u(t)), \  v_i(t) \in \partial f_i(u(t)) \quad \mbox{for almost all} \ t>0; \label{def1}\\
 &  \dot u(t) +  \eta (t) + \sum_i \theta_i (t) v_i(t)   = 0 \quad \quad \  \mbox{for almost all} \ t>0; \label{def2}\\
 &  \dot u(t) + ( N_K(u(t)) +    \mbox{{\rm Conv}}\left\{ \partial f_i(u(t)) \right\} )^{0}     = 0 \quad \mbox{for almost all} \ t>0.
 \end{align}
\end{definition} 

\subsection{Convergence properties}
}\fi

\begin{lemma}\label{Opial} Let   $S$ be
a non empty subset of $\mathcal H$, and  $u: [0, +\infty [ \to \mathcal H$ a map. Assume that
\begin{eqnarray*}
(i) &&\mbox{for every }z\in S,\>  \lim_{t \to  +\infty} 	\| u(t)- z 	\|  \mbox{ exists};\\
(ii) &&\mbox{every weak sequential cluster point of the map } u \mbox{ belongs to }S. 
\end{eqnarray*}
Then
 $$
w-\lim_{t \to  +\infty} u(t) = u_{\infty}  \   \   \mbox{ exists,  for some element
}u_{\infty}\in S. 
$$
\end{lemma}
\begin{proof} By $(i)$ and $S \neq \emptyset$, the trajectory $u$ is bounded in $\mathcal H$. In order to obtain its weak convergence, 
 we just need to prove that the trajectory has a unique weak sequential cluster point. Let
$u(t_n^1) \rightharpoonup z^1$ and $u(t_n^2) \rightharpoonup z^2$, with $t_n^1 \rightarrow +\infty$, and $t_n^2 \rightarrow +\infty$.
By $(ii)$, $z^1 \in S$, and $z^2 \in S$.  By $(i)$, it follows that $\lim_{t \to  +\infty} 	\| u(t)- z^1 	\|$ and  $\lim_{t \to  +\infty} 	\| u(t)- z^2 	\|$  exist.
Hence, $\lim_{t \to  +\infty} (	\| u(t)- z^1 	\|^2-	\| u(t)- z^2 	\|^2)$ exists. Developing and simplifying this last expression, we deduce that
\[
\lim_{t \to  +\infty} \left\langle  u(t), z^2 -z^1 \right\rangle  \  \ \mbox{exists}.
\]
Hence
\[
\lim_{n \to  +\infty} \left\langle  u(t_n^1), z^2 -z^1 \right\rangle =   \lim_{n \to  +\infty} \left\langle  u(t_n^2), z^2 -z^1 \right\rangle,
\]
which gives $\|z^2 -z^1\|^2=0$, and hence $z^2 = z^1$.
\end{proof}
 
\noindent We can now state our main convergence result.
\begin{theorem} 
\label{asymp1} 
Let us make assumptions {\rm H0), H1), H2)}. Then for any  strong global solution $t\in [0, +\infty [ \mapsto u(t) \in \mathcal H$ of {\rm(MOG)}, the following properties hold: 

\if{\noindent 	$i)$ \textit{Descent property}: \ for each $i=1,...,q$,  \   $t \mapsto f_i(u(t))$ is  a nonincreasing absolutely continuous function,  and for almost all $t>0$
\begin{equation}\label{asymp2} 
	\|\dot{u}(t)\|^2  + \frac{d}{dt}  f_i(u(t)) \leq 0.
\end{equation}
}\fi

\noindent 	 $i)$  \textit{Finite energy property}:
\begin{equation}\label{asymp3} 
		\int_0^{+\infty} \|\dot{u}(t)\|^2 dt < +\infty.
\end{equation}
\noindent 	 $ii)$ 	\textit{Weak convergence}:   Assume  that the trajectory 	$t\in [0, +\infty [ \mapsto u(t) \in \mathcal H$ is bounded in $\mathcal H$. Then
$u(t)$ converges weakly in $\mathcal H$	as $t\rightarrow +\infty$ to a weak Pareto optimum.
\end{theorem}
\begin{remark}\label{Pareto-opt} {\rm 
 a) Since each function $t \mapsto f_i(u(t))$ is  nonincreasing (see Proposition \ref{P:qualitative}), a natural condition insuring that the trajectory remains bounded  is that one of the functions $f_i$ has bounded  sublevel sets
 (see also Remark \ref{bounded orbit}).  \\
 b) Similarly, if  one of the functions $f_i$ has  relatively compact  sublevel sets (inf-compactness property), then the trajectory is relatively compact, and hence converges strongly 
 in $\H$.
It is an interesting  (open) question to find other conditions on the data ($f_i$ and $K$) which provide strong convergence of trajectories, and  extend the well-known conditions in the case of a single criterion.
}
\end{remark}
\begin{proof}
From Proposition \ref{P:qualitative} and by integrating (\ref{E:DecreaseProperty}), along with the fact that $f_i$ is bounded from below on $K$, we obtain
\begin{equation}\label{asymp9} 
		\int_0^{+\infty} \|\dot{u}(t)\|^2  dt  \leq f_i(u(0)) - {\inf}_K f_i.
\end{equation}
This proves items $i)$.

Let us now prove the weak convergence of any bounded trajectory $u$ of the (MOG) system. To that end we use  Opial's Lemma \ref {Opial}   with 
\begin{equation}\label{asymp10} 
	S= \left\{  v\in K: \  \forall i=1,...,q  \  \    f_i (v) \leq  \inf_{t\geq 0} f_i (u(t))             \right\}.
\end{equation}
Functions  $f_i$ are  convex continuous, and hence lower semicontinuous for the weak topology of $\mathcal H$. As well, the closed convex set $K$ is closed for the weak topology of $\mathcal H$.
The trajectory  $t\in [0, +\infty [ \mapsto u(t) \in \mathcal H$ has been assumed to be bounded in $\mathcal H$. As a consequence, every weak sequential cluster point of  the trajectory belongs to $S$, 
which is a closed convex non empty subset of $\mathcal H$.

$i)$ Take $z \in S$ and set, for any $t\geq 0$
\begin{equation}\label{asymp11}
h_z(t) = \frac{1}{2} \| u(t)- z 	\|^2.
\end{equation}
We have
\begin{equation}\label{asymp12}
\dot{h}_z(t) = \left\langle \dot{u}(t) , u(t)- z  \right\rangle .
\end{equation}
Since $u$ is a solution of (MOG), for almost all $t>0$ there exists 
\begin{equation}\label{asymp13}
\eta (t) \in  N_K(u(t)), \  \mbox{and}  \ v_i(t) \in \partial f_i(u(t)),\ \left(\theta_i(t)\right)\in \S^q \ 
\end{equation}
such that, 
\begin{equation}\label{asymp14}
  \dot u(t) +  \sum_{i=1}^{q} \theta_i (t) v_i(t)  + \eta (t)   = 0. \\
\end{equation}
By combining (\ref{asymp12}) and (\ref{asymp14}) we obtain
\begin{equation}\label{asymp15}
\dot{h}_z(t) + \sum_{i=1}^{q} \theta_i (t) \left\langle v_i(t)  , u(t)- z     \right\rangle +   \left\langle \eta (t) , u(t)- z      \right\rangle =0.
\end{equation}
On the one hand, since $\eta (t) \in  N_K(u(t))$ and $z \in K$
\begin{equation}\label{asymp16}
 \left\langle  \eta (t) ,  u(t)- z    \right\rangle \geq 0.
\end{equation}
On the other hand,  the convex subdifferential inequality at $u(t)$, and $v_i(t) \in \partial f_i(u(t))$ gives
\begin{equation}\label{asymp17}
f_i( z) \geq  f_i( u(t))  + \left\langle v_i(t)  , z - u(t)    \right\rangle.
\end{equation}
Since $z \in S$ we have $f_i( z) \leq  f_i( u(t))$, which gives
\begin{equation}\label{asymp18}
 \left\langle  v_i(t) ,  z-u(t)   \right\rangle \leq 0.
\end{equation}
As a consequence
\begin{equation}\label{asymp19}
 \sum_{i=1}^{q} \theta_i (t) \left\langle  v_i(t) ,  u(t) -z  \right\rangle \geq 0.
\end{equation}
Combining (\ref{asymp15}) with (\ref{asymp16}) and (\ref{asymp19}) we obtain
\begin{equation}\label{asymp20}
 \dot{h}_z(t) \leq 0.
\end{equation}
Hence, $h_z$ is a nonincreasing function, which proves item $i)$ of  Opial's Lemma \ref{Opial}.

Let us verify item  $ii)$ of Opial's Lemma \ref{Opial}.  Let $w-\lim u(t_n) = z$ for some sequence $t_n \rightarrow + \infty$. 
Since $u(t_n) \in K$ and $K$ is a closed convex subset of $\mathcal H$, we have $z \in K$. Moreover
\begin{align}\label{asymp21}
\inf_{t\geq 0} f_i (u(t)) &= \lim_{t\rightarrow + \infty} f_i (u(t))\\
                             &= \lim_{n\rightarrow + \infty} f_i (u(t_n))\\
                              & \geq f_i(z)
\end{align}
where the last inequality follows from the fact that $f_i$ is convex continuous, and hence lower semicontinuous for the weak topology of $\mathcal H$.
This being true for each $i=1,...,q$ we conclude that $z\in S$.
The two conditions of  Opial's Lemma \ref{Opial} are satisfied, which gives the weak convergence of each bounded trajectory of the (MOG) dynamic.
 Set
\begin{equation}\label{asymp210}
u(t) \rightharpoonup u_{\infty} \quad \mbox{weakly in} \ \mathcal H, \ \mbox{as } \ t \to + \infty,
\end{equation}
and show that  $u_{\infty}$    is  a Pareto critical point, and hence a weak Pareto optimum (Lemma \ref{equiv}).
The finite energy property (\ref{asymp3}) 
\begin{equation*}
\int_0^{+\infty} \|\dot{u}(t)\|^2 dt < +\infty
\end{equation*}
implies  
\begin{equation}\label{asymp202}
\mbox{liminfess}_{t \rightarrow + \infty } \|\dot{u}(t)\| =0.
\end{equation}
Since relations (\ref{def1}) and (\ref{def2}) are satisfied for almost  all $t>0$, (\ref{asymp202}) implies the existence of a sequence  $t_n \rightarrow + \infty $  such that
\begin{align}
&\dot{u}(t_n) \rightarrow 0 \  \mbox{strongly in  } \  \H  \label{asymp22} \\
& -\dot u(t_n)\in N_K(u(t_n)) + \mbox{{\rm Conv}}\left\{\partial f_i(u(t_n)) \right\}  \  \mbox{for each } \ n \in \mathbb N. \label{asymp220}
\end{align}
Moreover  by (\ref{asymp210})
\begin{equation}\label{asymp221}
u(t_n) \rightharpoonup u_{\infty} \quad \mbox{weakly in} \ \mathcal H.
\end{equation}
We conclude using (\ref{asymp22}), (\ref{asymp220}), (\ref{asymp221}),  and the following lemma which establishes a closure property for the operator $ N_K (\cdot)+ \mbox{{\rm Conv}}\left\{\partial f_i(\cdot) \right\}$.
\end{proof}

\begin{lemma}\label{L:ClosedOperator}
Under assumptions {\rm H0), H1)}, the multi-application
\begin{center}
$\begin{array}{rcl}
\mathcal{H} & \rightrightarrows & \mathcal{H}\\
u & \longmapsto & N_K(u) + \mbox{{\rm Conv}}\left\{\partial f_i(u) \right\}
\end{array}$
\end{center}
is demiclosed, i.e., its graph is sequentially closed for the  $\mbox{weak}-\mathcal H \times \mbox{strong}-\mathcal H $ topology.
\end{lemma}

\begin{proof}
Let $(u_n,w_n)$ be a sequence in the graph of  $N_K + \mathcal{C}$ where $\mathcal{C}(u)=\mbox{{\rm Conv}}\left\{\partial f_i(u) \right\}$. Suppose that $u_n$ converges weakly to  $\bar u \in K$, that $w_n$  converges strongly to $\bar w$, and prove that $\bar w \in N_K (\bar u) + \mathcal{C}(\bar u)$.
For each $n \in \N$,  there exists $q_n \in N_K(u_n)$, $p_{i,n} \in \partial f_i(u_n)$, $\lambda_{i,n} \in [0,1]$, $i=1,...,q$, such that
\begin{equation}\label{co1}
w_n = q_n + {\sum\limits_{i=1}^{q} \lambda_{i,n}p_{i,n}} ~~ \text{ \rm and } ~~ {\sum\limits_{i=1}^{q} \lambda_{i,n} = 1}.
\end{equation}
For each $n$, $(\lambda_{i,n})_{i=1,...,q}$ belongs to the unit simplex in ${\mathbb R}^q$, which is a compact set.  Hence we can extract a  subsequence (still noted $(\lambda_{i,n})$ to simplify the notation) such that, for each $i=1,...,q$ 
 \begin{equation}\label{asymp25}
\lambda_{i,n} \rightarrow  \bar \lambda_i,
\end{equation}
with
\begin{equation}\label{asymp26}
  0 \leq \bar \lambda_i \leq 1, \ \sum_{i=1}^{q} \bar \lambda_i=1.
\end{equation}
Noticing that the functions $f_i$ are convex continuous, thanks to the Moreau-Rockafellar additivity rule for the subdifferential of a sum of convex functions, we can rewrite
(\ref{co1}) as follows
\begin{equation}\label{co2}
w_n \in \partial \left( \delta_K + \sum\limits_{i=1}^{q} \lambda_{i,n}  f_i \right)(u_n),
\end{equation}
where $\delta_K$ is the indicator function of $K$.
Equivalently, for any $\xi \in \H$
\begin{equation}\label{co3}
\sum\limits_{i=1}^{q} \lambda_{i,n}  f_i (\xi)  + \delta_K (\xi) \geq \sum\limits_{i=1}^{q} \lambda_{i,n}  f_i (u_n)  + \delta_K (u_n)  +  \left\langle w_n  , \xi - u_n   \right\rangle.
\end{equation}
Let us pass to the lower limit 
in (\ref{co3}). By using (\ref{asymp25}), the lower semicontinuity property of the $f_i$ and $\delta_K$ for the weak topology of $\H$ ($K$ is closed convex and hence weakly closed), and the weak (resp. strong) convergence of $u_n$ (resp. $w_n$), we obtain
\begin{equation}\label{co4}
\sum\limits_{i=1}^{q} \bar \lambda_i  f_i (\xi)  + \delta_K (\xi) \geq \sum\limits_{i=1}^{q} \bar \lambda_i  f_i (\bar u)  + \delta_K (\bar u)  +  \left\langle \bar w  , \xi - \bar u   \right\rangle.
\end{equation}
In the above limit process, we use the fact that the functions $f_i$ are finitely valued (otherwise we would face the delicate question concerning the product $0 \times \infty$).
Using again the Moreau-Rockafellar additivity rule, we equivalently obtain
 \[
 \bar w  \in \sum_{i=1}^{q} \bar \lambda_i  \partial f_i(\bar u)  + N_K (\bar u),
 \]
 which, with (\ref{asymp26}),  expresses that $( \bar u, \bar w)$   is in the graph of  $N_K + \mathcal{C}$.
\end{proof}

\begin{remark}\label{bounded orbit} {\rm 
Suppose that there exists an ideal solution $\bar{z}$ to (CMO). Then, for any solution trajectory of (MOG), $\bar{z} \in S$, where $S$ has been defined in (\ref{asymp10}). 
Following the proof of  Theorem \ref{asymp1}, the function  $h_{\bar{z}}(\cdot)= \frac{1}{2} \| u(\cdot)- \bar{z} 	\|^2$ is nonincreasing. Thus, in that case,  any trajectory of (MOG) is bounded.
We recover the   fact that, in the case of a single convex objective function, the trajectories of the steepest descent equation are bounded iff the solution set is not empty.
}
\end{remark}

\section{Existence of strong global solutions}\label{ex-yo}

\noindent In this section, it is assumed that $  \mathcal H =  \mathbb R^d $ is a finite dimensional Euclidean space. 
This is because our proof of the existence of solutions to the (MOG) dynamic is based on the Peano theorem, and not on the Cauchy-Lipschitz.
It is likely that the proof can be adapted to the case of infinite dimension by making
ad hoc assumption on the data (as inf-compactness).
This is an interesting topic for further studies, particularly involving applications to PDEs.
Our approach is based on the regularization of the non-smooth functions $f_i$  by the Moreau-Yosida approximation. 
This approximation brings us back to the situation studied in \cite{AG1}, which considers the case of differentiable functions. 

\subsection{Statement of the result}
\begin{theorem} \label{basic-exist} 
 Let $  \mathcal H $ be a finite dimensional Hilbert space. Let us make assumptions {\rm H0), H1), H2)}.
Then, for any initial data $u_0\in K$, there exists a strong global solution $u:  [0,+\infty [ \rightarrow \mathcal H$ of {\rm(MOG)} system {\rm(\ref{mult-steep-desc})},
which satisfies $u(0)= u_0$.
\end{theorem}

\begin{remark}
{\em In Theorem \ref{basic-exist}, for any $u_0 \in K$,
we claim the existence of a strong global solution $u:  [0,+\infty [ \rightarrow \mathcal H$ of {\rm(MOG)} system,  satisfying the Cauchy data $u(0)=u_0$.
 By definition of a strong solution, $u$ is absolutely continuous
on any finite time interval $[0,T]$, but from Proposition \ref{P:qualitative} we know it is  moreover Lipschitz continuous.    }
\end{remark}

\noindent In the above theorem, we only claim  existence. Without further assumptions, uniqueness is not guaranteed. 
Indeed, the following proof of existence relies on Peano,  not  Cauchy-Lipschitz theorem. Before entering the proof of  existence, we will briefly discuss the question of uniqueness which remains an open question.
 
\begin{remark}\label{uniqueness} {\rm  
 In the unconstrained case, and for convex differentiable objective functions, illustrative examples of the (MOG) dynamic were given in Section \ref{S:Examples}. In these elementary situations,  we have been able to explicitely compute  the vector field $v\mapsto s(v)$. We observed that it can be Lipschitz continuous (Example 1 and 2) or only H\"older continuous (Example 3).
This naturally raises the following question: 
in the unconstrained case, and for differentiable objective functions, what are the assumptions ensuring that the vector field $v\mapsto s(v)$ is Lipschitz continuous (recall that it is H\"older continuous, see \cite{AG1})? This is clearly a key property for uniqueness for (MOG).

}
\end{remark}

The end of this section is devoted to the proof of Theorem \ref{basic-exist}, which is quite technical. To make reading easier, the proof has been divided into several stages.
First of all, let us bring some additional results to \cite{AG1}, which concern the smooth case, and which will be useful for our study.

\subsection{The smooth case, complements}
Let us suppose that the $f_i$ are convex differentiable functions. Following  \cite{AG1}, for any $u_0\in K$, there exists a strong global solution $u: [0,+\infty[ \rightarrow \mathcal H$ of the Cauchy problem
\begin{equation}\label{bas100}
\begin{cases}
 \dot u(t) + \bigg(N_K(u(t)) + \mbox{{\rm Conv}}\left\{\nabla f_i(u(t)) \right\} \bigg)^{0}= 0,\\
u(0)= u_0.
\end{cases}
\end{equation}
The concept of solution $ u $ is as follows.

$\left(i\right)$ $u: [0,+\infty[ \rightarrow \mathcal H$
is  absolutely continuous on each  interval
$\left[0,T\right]$, $0<T<+\infty;$

 $\left(ii\right)$ there exists $\eta: [0,+\infty[ \rightarrow \mathcal H$ and $w: [0,+\infty[ \rightarrow \mathcal H$ which satisfy
 \begin{align}
 & \eta\in L^2(0,T; \mathcal H), \quad w\in L^{\infty}(0,T; \mathcal H) \quad \mbox{for all} \ T>0; \label{bas201}\\
 & \eta(t) \in N_K(u(t)), \  w(t) \in \mbox{{\rm Conv}}\left\{\nabla f_i(u(t)) \right\}   \quad \mbox{for almost all} \ t>0; \label{bas202}\\
 & \eta(t) + w(t) = \bigg( N_K(u(t)) + \mbox{{\rm Conv}}\left\{\nabla f_i(u(t)) \right\}) \bigg)^{0} \quad \mbox{for almost all} \ t>0; \label{bas203}\\
 &  \dot u(t) + \eta(t) + w(t)= 0 \quad \mbox{for almost all} \ t>0. \label{bas204}
 \end{align}
Let us make precise (\ref{bas202}).

\begin{lemma}\label{basic-sel} Let $u$ be a solution of {\rm(\ref{bas100})}, and  $\eta,  w$ the associated functions satisfying {\rm(\ref{bas201})-(\ref{bas202})-(\ref{bas203})-(\ref{bas204})}. Then
 $w(t) \in \mbox{{\rm Conv}}\left\{\nabla f_i(u(t)) \right\}$ can be written as follows:
\begin{equation}
\label{bas301}
w(t)= \sum_i \theta_i(t) \nabla f_i(u(t))
\end{equation}
with
 $\theta_{i}\in L^{\infty}(0, + \infty) $, $i=1,2,...,q$,  \  and for almost all $t>0, \ (\theta_i(t))\in \S^q$.
\if{Thus, equation {\rm(\ref{bas204})}  becomes
\begin{equation}\label{bas303}
  \dot{u}(t) +  \sum_{i=1}^{q} \theta_{i} (t) \nabla f_{i}(u(t))   + \eta (t)  = 0. 
\end{equation}}\fi
\end{lemma}
\begin{proof} From  (\ref{bas203}) we see that for almost all $t>0$
\begin{equation}\label{bas304}
  w(t) = \mbox{proj}_{\mathcal{C}(u(t))} (-\eta(t)),
\end{equation}
where $\mathcal{C}(u(t))= \mbox{{\rm Conv}}\left\{\nabla f_i(u(t)) \right\}$.
Equivalently, $ w(t) =\sum_i \theta_i(t) \nabla f_i(u(t))$ for any  $ \theta (t)=(\theta_i (t))$ such that
\begin{equation}\label{bas305}
  \theta(t) \in \mbox{argmin} \left\{j(t, \theta): \  \theta \in \mathbb R^q              \right\}
\end{equation}
where
\begin{equation}\label{bas306}
 j(t, \theta) = \| \eta(t) + \sum_i \theta_i   \nabla f_{i}(u(t))  \|   + \delta_{\S^q} (\theta),          
\end{equation}
where $\delta_{\S^q}$ is the indicator function of $\S^q$. 
The crucial point is to prove that we can take the  $\theta_i (t)$ measurable.
Since $j: [0,+\infty[ \times {\mathbb R}^q \rightarrow \mathbb R \cup \left\{+\infty\right\}$ is a positive (convex) normal integrand, the mapping 
$ t \mapsto \mbox{argmin} j(t, \cdot)$ is measurable, and hence admits a measurable selection $t \mapsto \theta(t)$, see \cite[Corollary 14.6; Theorem 14.37]{RW}.
Hence, we can write $w(t)= \sum_i \theta_i(t) \nabla f_i(u(t))$, with $\theta_i$ measurable, and $\theta (t) \in \S^q$. Since $\theta_i$ is bounded, we have 
$$w(t)= \sum_i \theta_i(t) \nabla f_i(u(t)) \  \mbox{and} \ \theta_{i}\in L^{\infty}(0, + \infty), \ \theta(t) \in \S^q \ \mbox{a.e.}\  t>0. $$\end{proof}
Let us now return to our setting involving non-smooth objective functions $f_i$.

\subsection{Approximate equations}

The main difficulty  comes from the discontinuity of the vector field which governs the (MOG) dynamic (\ref{mult-steep-desc}).
 As a main ingredient of our approach, we use the Moreau-Yosida approximation of the convex functions $f_i$ (equivalently the Yosida approximation of the maximal monotone operators $\partial f_i$), $i=1,...,q$. This regularization method  is widely used in  nonsmooth convex analysis, see \cite{Att00}, \cite{AE}, \cite{BC},   \cite{Br}, \cite{Z} for a detailed presentation.  Its main properties are summarized in the following statement.
 
\begin{proposition}\label{MY} Let $\Phi: \mathcal H \rightarrow \mathbb R \cup \left\{+ \infty \right\}$  be a closed convex proper function. 
The Moreau-Yosida approximation of index $\lambda  >0$  of $\Phi$ is the function ${\Phi}_{\lambda}: \mathcal H \rightarrow \mathbb R$ which is defined 
for all $\ v\in \mathcal H$  by 
	\begin{equation}\label{mult-gsd3}
	    {\Phi}_{\lambda}(v) = \inf \left\{ \Phi (\xi)  +    \frac{1}{2 \lambda}  \|v - \xi\|^2:  \quad \xi \in   \mathcal H          \right\}.  
\end{equation}
\begin{enumerate}\label{mult-gsd4}
	\item The infimum in {\rm(\ref{mult-gsd3})} is attained at a unique point $J_{\lambda}v \in \mathcal H$, which satisfies		
\begin{align} \label{MY1}
	&{\Phi}_{\lambda}(v) =  \Phi (J_{\lambda}v)  +    \frac{1}{2 \lambda}  \|v - J_{\lambda}v\|^2;\\
	 & J_{\lambda}v   +   \lambda   \partial \Phi ( J_{\lambda}v) \ni v.  
\end{align}		
$J_{\lambda}= (I + \lambda   \partial \Phi)^{-1}: \mathcal H \rightarrow \mathcal H$ is everywhere defined and nonexpansive. It is  called the resolvent of index $\lambda$ of $A= \partial \Phi$.
	\item ${\Phi}_{\lambda}$  is   convex,  and continuously differentiable. Its gradient at $v \in  \mathcal H$ is equal to
	\begin{equation}\label{mult-gsd5}
	\nabla{\Phi}_{\lambda}(v) =     \frac{1}{\lambda}  (v - J_{\lambda}v).
\end{equation}
	\item  The operator $A_{\lambda}= \nabla{\Phi}_{\lambda}=\frac{1}{\lambda}  (I - J_{\lambda})$ is called the Yosida approximation of index $\lambda$ of the maximal monotone operator $A= \partial \Phi$. It is Lipschitz continuous with Lipschitz constant $\frac{1}{\lambda} $.
	\item For any $v \in \mbox{dom}A$, \   $\|A_{\lambda}v   \| \leq \|A^0 (v )  \|$, ($A^0 (v )$ is the element of minimal norm of $A(v)$).
	\item For any $v \in \mathcal H$,  ${\Phi}_{\lambda}(v) \uparrow \Phi (v)$ \ as \ $\lambda \downarrow 0$.
\end{enumerate}
\end{proposition}
 
We are going to adapt to our situation the classical proof of the existence of strong solutions to evolution equations governed by subdifferentials of convex lower semicontinuous functions, see \cite{Br}.
For each $\lambda >0$, we set $f_{i,\lambda} = (f_i)_{\lambda}$ the Moreau-Yosida approximation of index $\lambda$  of $f_i$. We
consider the  Cauchy problem which is obtained by replacing  each $\partial f_i$ by its Yosida approximation $\nabla f_{i,\lambda}$, in (MOG).  
So doing, we are in the situation studied in \cite{AG1}, which treats the case of differentiable objective functions. Precisely, by \cite[Theorem 3.5]{AG1}, for each $\lambda >0$
there  exists of a strong global solution
$u_{\lambda}: [0,+\infty[ \rightarrow \mathcal H$ of the Cauchy problem
\begin{equation}
\label{app1} {\rm\mbox{(MOG)}}_{\lambda} 
\begin{cases}
 \dot u_{\lambda}(t) + \bigg(N_K(u_{\lambda}(t)) + \mbox{{\rm Conv}}\left\{\nabla f_{i,\lambda}(u_{\lambda}(t)) \right\} \bigg)^{0}= 0,\\
u_{\lambda}(0)= u_0.
\end{cases}
\end{equation}
By Lemma \ref{basic-sel} and (\ref{bas204}), there exists $\theta_{i, \lambda} \in L^{\infty}(0, + \infty)$, and ${\eta}_{\lambda} \in  L^{2}(0,T; \mathcal H)$ for all $T>0$,  such that, for almost all $t>0$
\begin{equation}\label{app3}
  {\dot{u}}_{\lambda}(t) +  \sum_{i=1}^{q} \theta_{i, \lambda} (t) \nabla f_{i,\lambda}(u_{\lambda}(t))   + {\eta}_{\lambda} (t)  = 0,
\end{equation}
and 
\begin{equation}\label{app2}
{\eta}_{\lambda} (t) \in  N_K(u_{\lambda}(t)), \   (\theta_{i, \lambda} (t)) \in \S^q. 
\end{equation}

\subsection{Estimations on the sequence $(u_{\lambda})$}

Let us establish bounds for the net $(u_{\lambda})_{\lambda}$, which are independent of $\lambda$.
Let us make a similar argument to that used in Theorem \ref{asymp1}, just replacing $f_i$ by $f_{i,\lambda}$. We obtain
\begin{equation}\label{est1} 
		\int_0^{+\infty} \|{\dot{u}}_{\lambda}(t)\|^2  dt  \leq f_{i,\lambda}(u_0) - {\inf}_{\mathcal H} f_{i,\lambda}.
\end{equation}
Then notice that $f_{i,\lambda}(u_0)  \leq f_{i}(u_0)$, and $\inf_{\mathcal H} f_{i,\lambda} = \inf_{\mathcal H} f_{i}$.  Hence
\begin{equation}\label{est2} 
		\int_0^{+\infty} \|{\dot{u}}_{\lambda}(t)\|^2  dt  \leq f_{i}(u_0)  - \inf_{\mathcal H} f_{i},
\end{equation}
and 
\begin{equation}\label{est3} 
	\sup_{\lambda}	\int_0^{+\infty} \|{\dot{u}}_{\lambda}(t)\|^2  dt  < + \infty.
\end{equation}
From 
\begin{equation}\label{est4} 
	u_{\lambda}(t) = u_0 + \int_0^t  \dot u_{\lambda}(\tau)d\tau,
\end{equation}
and Cauchy-Schwarz inequality, we obtain
\begin{equation}\label{est5} 
\|	u_{\lambda}(t) \| \leq \| u_0 \| + \sqrt{t}\left(\int_0^t  \|\dot u_{\lambda}(\tau)\|^2 d\tau \right)^{\frac{1}{2}}.
\end{equation}
 Combining (\ref{est3})  with (\ref{est5}) we deduce that, for any $T>0$
\begin{equation}\label{est6} 
	\sup_{\lambda}	 {\|{u}_{\lambda}\|}_{L^{\infty}([0,T]; \mathcal H)}   < + \infty.
\end{equation}
Let us now consider the gradients terms $\nabla f_{i,\lambda}(u_{\lambda})$ which appear in (\ref{app3}).
By Proposition \ref{MY}, item 4.,  for any $v\in \mathcal H$,  $\lambda >0$, and $i=1,2,...,q$ 
\begin{equation}\label{est7} 
		 \|\nabla f_{i,\lambda}  (v)\|  \leq \|\left(\partial f_{i}\right)^{0}   (v)\|.
\end{equation}
Combining (\ref{est6}) with (\ref{est7}), and using assumption H1), which tells us that $f_i$  is  a convex continuous function whose subdifferential $\partial f_i$ is bounded on bounded sets, we obtain that, for any $T>0$
\begin{equation}\label{est8} 
	\sup_{\lambda}	 \|\nabla f_{i,\lambda}({u}_{\lambda})\|_{L^{\infty}([0,T]; \mathcal H)}   < + \infty.
\end{equation}

\subsection{Passing to the limit ($\lambda \to 0$)}
As we have already pointed out, the difficulty comes from the discontinuous nature 
of the multivalued operators $\partial f_i$ and $N_K$, and hence of the vector field which governs the differential equation (\ref{mult-steep-desc}).
Indeed, we are going to use  the   monotonicity property of these operators, and the
demiclosedness property (closedness for the strong $\times$ weak product topology) of their graphs in the associated functional spaces.

 By  (\ref{est3}), (\ref{est6}), the generalized sequence $(u_{\lambda})$ is uniformly bounded and equi-continuous on $[0,T]$.  
Since $\mathcal H$ is finite dimensional, we deduce from Ascoli's theorem that, for any $0< T <+\infty$, the generalized sequence $(u_{\lambda})$ is relatively compact for the uniform convergence topology on $[0,T]$.
 Thus, by a diagonal argument (we keep the notation $(u_{\lambda})$ for simplicity), we obtain the existence of $u \in \mathcal C([0,+\infty[ ; \mathcal H)$, 
 and $v_i ,  \eta \in L^{2}_{loc}(0,+\infty ; \mathcal H)$, $\theta_i \in L^{\infty}(0,+ \infty)$ such that, for any $0< T <+\infty$,
\begin{align}
&&&u_{\lambda} \rightarrow u          &&  \mbox{strong}-\mathcal C(0,T ; \mathcal H) & \label{lim01}\\ 
&&& \dot u_{\lambda}  \rightharpoonup \dot{u}      &&\mbox{weak}-L^{2}(0,T ; \mathcal H) & \\
&&&  \nabla f_{i,\lambda}({u}_{\lambda})   \rightharpoonup   v_i        &&     \sigma(L^{\infty}(0,T; \mathcal H), L^{1}(0,T; \mathcal H))  &\\
&&& \theta_{i, \lambda}   \rightharpoonup   \theta_i    \   && \sigma(L^{\infty}(0,T), L^{1}(0,T))& \\
&&&  \eta_{\lambda}   \rightharpoonup   \eta    \ \  && \mbox{weak}-L^{2}(0,T ; \mathcal H). & 
\end{align}
The last statement comes from the following observation: by (\ref{app3})
\begin{equation}\label{lim2}
  {\eta}_{\lambda} (t)  =  -{\dot{u}}_{\lambda}(t) - \sum_{i=1}^{q} \theta_{i, \lambda} (t) \nabla f_{i,\lambda}(u_{\lambda}(t)),
\end{equation}
which implies that the net $({\eta}_{\lambda})$   remains bounded in $L^{2}(0,T ; \mathcal H)$ for any $T>0$.\\
Let us complete this list with the convergence of the net $(f_{i,\lambda}({u}_{\lambda}))$.
\begin{lemma}\label{lim06}
The following convergence result holds: for any $0< T <+\infty$
\begin{equation}\label{lim6}
 f_{i,\lambda}(u_{\lambda})  \rightarrow  f_{i}(u) \ \mbox{uniformly on} \  [0, T] \  \mbox{as}  \  \lambda \rightarrow 0.
\end{equation}
\end{lemma}
\begin{proof}
Let us fix $T>0$, and work on the bounded interval $[0,T]$. Let us  write the triangle inequality
\begin{equation}\label{lim7}
 | f_{i,\lambda}(u_{\lambda})  -  f_{i}(u) | \leq | f_{i,\lambda}(u_{\lambda})  -  f_{i,\lambda}(u) |   + | f_{i,\lambda}(u)  -  f_{i}(u) |.
\end{equation}
On the one hand, by the Mean value theorem, (\ref{est7}), and (\ref{lim01}) 
\begin{align}\label{lim8}
| f_{i,\lambda}(u_{\lambda}(t))  -  f_{i,\lambda}(u(t)) |  & \leq  \left( \sup_{\xi \in [u_{\lambda}(t),u(t)]}  \|\nabla f_{i,\lambda}(\xi)\|\right) \|u_{\lambda}(t) - u(t)   \| \\
& \leq  \left( \sup_{\xi \in [u_{\lambda}(t),u(t)]}  \|\left(\partial f_{i}\right)^{0} (\xi)\|\right) \|u_{\lambda}(t) - u(t)   \| \\
 & \leq C \|u_{\lambda}(t) - u(t)   \|,
\end{align}
and hence, 
\begin{equation}\label{lim9}
f_{i,\lambda}(u_{\lambda})  -  f_{i,\lambda}(u) \rightarrow 0 \  \   \ \mbox{uniformly on} \  [0, T] \  \mbox{as}  \  \lambda \rightarrow 0.  
\end{equation}
On the other hand, the net $(f_{i,\lambda}(u))_{\lambda}$  is equi-continuous. This results from the following inequalities
\begin{align}\label{lim90}
|\frac{d}{dt} f_{i,\lambda}(u(t)) |  & = |\left\langle   \nabla f_{i,\lambda}(u(t)), \dot{u} (t)   \right\rangle |\\
& \leq   \|\left(\partial f_{i} (u(t))\right)^{0}\| \|\dot{u} (t) \|   \\
 & \leq C \|\dot{u} (t) \|,
\end{align}
and
\begin{align}\label{lim91}
|f_{i,\lambda}(u(t)) - f_{i,\lambda}(u(s)) |  & \leq  \int_s^t  | \frac{d}{d \tau} f_{i,\lambda}(u(\tau)) | d\tau\\
& \leq   \sqrt{t-s}  \left(\int_0^T  | \frac{d}{d \tau} f_{i,\lambda}(u(\tau)) |^2 d\tau \right)^{\frac{1}{2}} \\
 & \leq  C \sqrt{t-s}  \left(\int_0^T  \|\dot{u} (\tau) \|^2  d\tau \right)^{\frac{1}{2}}. 
\end{align}
Hence, the net $(f_{i,\lambda}(u))_{\lambda}$   is equi-continuous. Since it converges pointwise to   $f_{i}(u)$, 
by Ascoli Theorem, we obtain
\begin{equation}\label{lim10}
f_{i,\lambda}(u)  -  f_{i}(u) \rightarrow 0 \  \  \ \mbox{uniformly on} \  [0, T] \  \mbox{as}  \  \lambda \rightarrow 0.  
\end{equation}
Combining (\ref{lim7}), (\ref{lim9}), (\ref{lim10}), we obtain (\ref{lim6}).
\end{proof}

Technically, the most difficult point is to pass to the limit in (\ref{app3}) on the product of the two weakly converging sequences $(\theta_{i, \lambda})$ and 
$(\nabla f_{i,\lambda}(u_{\lambda}))$. In order to circumvent this difficulty,
   we use a variational argument based on  the convex differential inequality: for any $\xi \in L^{\infty}(0,T ; \mathcal H)$, 
\begin{equation}\label{lim3}
 \sum_{i=1}^{q} \theta_{i, \lambda} (t) f_{i,\lambda} (\xi(t))   \geq    \sum_{i=1}^{q} \theta_{i, \lambda} (t)   f_{i,\lambda}(u_{\lambda}(t))   +    \left\langle     \sum_{i=1}^{q} \theta_{i, \lambda} (t) \nabla f_{i,\lambda}(u_{\lambda}(t)),  \xi(t) -  u_{\lambda}(t)  \right\rangle.
\end{equation}
After integration on  $[0,T]$, we obtain 
\begin{align}
 \int_0^T  \sum_{i=1}^{q} \theta_{i, \lambda} (t) f_{i,\lambda} (\xi(t)) dt  & \geq    \int_0^T   \sum_{i=1}^{q} \theta_{i, \lambda} (t)   f_{i,\lambda}(u_{\lambda}(t)) dt  \label{lim4}\\
 &+  \int_0^T   \left\langle     \sum_{i=1}^{q} \theta_{i, \lambda} (t) \nabla f_{i,\lambda}(u_{\lambda}(t)),  \xi(t) -  u_{\lambda}(t)  \right\rangle dt. \label{lim40}
\end{align}
By  (\ref{app3}),   $\sum_{i=1}^{q} \theta_{i, \lambda} (t) \nabla f_{i,\lambda}(u_{\lambda}(t)) =     -{\dot{u}}_{\lambda}(t) - {\eta}_{\lambda} (t) $. Replacing  in (\ref{lim4})-(\ref{lim40}), we obtain
\begin{align}
 \int_0^T  \sum_{i=1}^{q} \theta_{i, \lambda} (t) f_{i,\lambda} (\xi(t)) dt  & \geq    \int_0^T   \sum_{i=1}^{q} \theta_{i, \lambda} (t)   f_{i,\lambda}(u_{\lambda}(t)) dt  \label{lim5}\\
 &+  \int_0^T   \left\langle   -{\dot{u}}_{\lambda}(t) - {\eta}_{\lambda} (t),  \xi(t) -  u_{\lambda}(t)  \right\rangle dt.\label{lim50}
\end{align}
Since $ f_{i,\lambda} (\xi(t)) \leq  f_{i} (\xi(t))$, and $\theta_{i, \lambda} (t) \geq 0$, we obtain
\begin{align}
 \int_0^T  \sum_{i=1}^{q} \theta_{i, \lambda} (t) f_{i} (\xi(t)) dt  & \geq    \int_0^T   \sum_{i=1}^{q} \theta_{i, \lambda} (t)   f_{i,\lambda}(u_{\lambda}(t)) dt  \label{lim51}\\
 &+  \int_0^T   \left\langle   -{\dot{u}}_{\lambda}(t) - {\eta}_{\lambda} (t),  \xi(t) -  u_{\lambda}(t)  \right\rangle dt.\label{lim501}
\end{align}
 For any $\xi \in L^{\infty}(0,T ; \mathcal H)$, since $f_i$ is continuous and bounded on bounded sets (assumption H1)), we have    $f_{i} (\xi(\cdot))  \in L^{\infty}(0,T)$. Moreover $\theta_{i, \lambda}   \rightharpoonup   \theta_i$    for the topology  $\sigma(L^{\infty}(0,T), L^{1}(0,T))$. Therefore, by passing to the limit on the left member of (\ref{lim51}), we obtain
$$
\lim_{\lambda} \int_0^T  \sum_{i=1}^{q} \theta_{i, \lambda} (t) f_{i} (\xi(t)) dt = \int_0^T  \sum_{i=1}^{q} \theta_{i} (t) f_{i} (\xi(t)) dt.
$$
\noindent Let us now pass to the limit on the right member of (\ref{lim51})-(\ref{lim501}). For the first term,
we use Lemma \ref{lim06}.  For the second term, we notice that this expression involves  duality products of nets which are  respectively converging for the strong and weak topologies of a duality pairing. 
More precisely $ {\dot{u}}_{\lambda} + {\eta}_{\lambda}$ converges $\mbox{weakly in }L^{2}(0,T ; \mathcal H)$ to $\dot{u} + \eta$, and  $\xi -  u_{\lambda} $ converges uniformly, and hence 
$\mbox{strongly in } L^{2}(0,T ; \mathcal H)$ to $\xi -  u$.
Hence, by passing to the limit as $\lambda $ goes to zero, we obtain
\begin{align}\label{lim11}
 \int_0^T  \sum_{i=1}^{q} \theta_i (t) f_i (\xi(t)) dt  &\geq    \int_0^T   \sum_{i=1}^{q} \theta_i (t)   f_i (u(t)) dt \\ 
 &+  \int_0^T   \left\langle   -\dot{u}(t) - \eta (t),  \xi(t) -  u(t)  \right\rangle dt. \nonumber
\end{align}
Let us interpret this inequality in the duality pairing bewteen  the functional spaces $L^{\infty}(0,T ; \mathcal H)$ and $L^{1}(0,T ; \mathcal H) \subset \left(L^{\infty}(0,T ; \mathcal H)\right)^{*}$. For this, introduce
$I$,  the integral functional on  $L^{\infty}(0,T ; \mathcal H)$ which is defined by
\begin{equation}\label{int1}
 I (\xi ) = \int_0^T  \sum_{i=1}^{q} \theta_i (t) f_i (\xi(t)) dt.
\end{equation}
We observe that $I: L^{\infty}(0,T ; \mathcal H) \rightarrow \mathbb R $ is  convex and continuous on $L^{\infty}(0,T ; \mathcal H)$. Hence, inequality (\ref{lim11}) can be rewritten as
\begin{equation}\label{Rock+}
-\dot{u} - \eta  \in \partial I (u).
\end{equation}
According to the duality theorem of Rockafellar  for convex functional integrals, see \cite[Theorem 4]{Rock}, for almost all $t>0$
 \begin{align}\label{lim141}
 -\dot{u}(t) - \eta (t)  &\in  \partial \left(\sum_{i=1}^{q} \theta_i (t) f_i \right)(u(t))  \\
 & =  \sum_{i=1}^{q} \partial \left(\theta_i (t) f_i \right)(u(t)),
\end{align}
where the last equality comes from the additivity rule for the subdifferential of the sum of convex continuous functions on ${\mathbb R}^q$.
Indeed we need to prove a slighter more precise result: 
$$
-\dot{u}(t) - \eta (t) = \sum_{i=1}^{q}  \theta_i (t) v_i (t)
$$ 
with \textit{measurable} functions $v_i \in L^{\infty}(0,T; \mathcal H)$ such that
\begin{equation}\label{Rock++}
v_i(t) \in  \partial  f_i (u(t)) \ \text{ for almost all } t>0.
\end{equation}
This can be proved by a precise analysis of the duality theorem from \cite{Rock}. Since it is quite technical, the proof is stated in Lemma \ref{L:Rock}, at the end of this section.  
Assuming this result, we obtain by combination with (\ref{Rock+}) that 
 \begin{equation}\label{lim14}
\dot u(t) +  \eta (t) + \sum_i \theta_i (t) v_i(t)   = 0 \quad \mbox{for almost all} \ t>0,
\end{equation}
with
 \begin{align}
 & \theta_i \in L^{\infty}(0,+\infty; \R), \ v_i\in L^{\infty}(0,T; \mathcal H), \ \eta \in L^2(0,T; \mathcal H), \ \mbox{for all} \ T>0,   \  \mbox{and all} \  i=1,2,..., q; \\
 &  (\theta_i(t))\in \S^q \ \text{and} \  v_i(t) \in \partial f_i(u(t)) \quad \mbox{for almost all} \ t>0;  \label{lim112}
 \end{align}
On the other hand, from $u_{\lambda} \rightarrow u     \   \mbox{strong}-\mathcal C(0,T ; \mathcal H) $, 
$\eta_{\lambda}   \rightharpoonup   \eta     \  \  \mbox{weak}-L^{2}(0,T ; \mathcal H)$, \ ${\eta}_{\lambda} (t) \in  N_K(u_{\lambda}(t))$, and from the demi-closedness property of the extension
to $L^2(0,T; \mathcal H)$ of the maximal monotone 
normal cone mapping ($N_K$ is the subdifferential of the indicator function fo $K$), we obtain 
 \begin{equation}\label{lim15}
\eta (t) \in N_K (u(t)).
\end{equation}
Thus
\begin{equation}\label{lim17}
 \dot u(t) + N_K(u(t)) + \mbox{Conv}\left\{\partial f_i(u(t))\right\} \ni 0.
\end{equation}

\subsection{Lazy solution}

Let us complete the proof of Theorem \ref {basic-exist} by showing that $u$ is a lazy solution of the differential inclusion  (\ref{lim17}).
Let us start from the lazy solution property satisfied by the approximate solutions $u_{\lambda}$
\begin{equation}\label{sl1}
 -\dot u_{\lambda}(t) = \bigg(N_K(u_{\lambda}(t)) + \mbox{{\rm Conv}}\left\{\nabla f_{i,\lambda}(u_{\lambda}(t)) \right\} \bigg)^{0}.
\end{equation}
By the obtuse angle property, since $0 \in N_K(u_{\lambda}(t))$  we have 
\begin{equation}\label{sl2}
\left\langle  \dot u_{\lambda}(t), \dot u_{\lambda}(t)  +  \sum_{i=1}^{q} \theta_{i}(t) \nabla f_{i,\lambda} (u_{\lambda}(t)) \right\rangle  \leq 0,
\end{equation}
for all $\theta_{i}  \in L^{\infty}(0,+\infty), \  i=1,2,...,q $ that satisfy $(\theta_i(t))\in \S^q$.
After developing, and using the classical derivation chain rule, we obtain
\begin{equation}\label{sl4}
\| \dot u_{\lambda}(t)\|^2 +  \sum_{i=1}^{q}  \theta_{i}(t) \frac{d}{dt} f_{i,\lambda} (u_{\lambda}(t)) \leq 0.
\end{equation}
In order to pass to the limit on (\ref{sl4}), 
take  $\alpha$ a  nonnegative test function (a function of $t$ which is regular, and with compact support in $]0,  T[$). After multiplication of (\ref{sl4}) by $\alpha$, and integration on 
$[0,T]$, we obtain 
\begin{equation}\label{sl5}
\int_0^T \alpha (t)\| \dot u_{\lambda}(t)\|^2 dt + 
\sum_{i=1}^{q}  \int_0^T \alpha (t) \theta_{i}(t) \frac{d}{dt} f_{i,\lambda} (u_{\lambda}(t)) dt \leq 0.
\end{equation}
The convex function $ v \mapsto \int_0^T \alpha (t)\| v(t)\|^2 dt$ is continuous on $L^{2}(0,T ; \mathcal H)$, and hence lower semicontinuous
 for the weak topology of $L^{2}(0,T ; \mathcal H)$. Since $\dot u_{\lambda}  \rightharpoonup \dot{u}$  weakly in $L^{2}(0,T ; \mathcal H)$, we have
\begin{equation}\label{sl6}
\int_0^T \alpha (t)\| \dot u(t)\|^2 dt \leq  \mbox{{\rm lim inf}} \int_0^T \alpha (t)\| \dot u_{\lambda}(t)\|^2 dt.
\end{equation}
In order to pass to the limit on the second term of (\ref{sl5}), we use a density argument. First assume that the $\theta_{i}$  are Lipschitz continuous on bounded sets.
Since $f_{i,\lambda} (u_{\lambda})$ and $\alpha \theta_{i}$ are absolutely continuous functions of a real variable, their product is still absolutely continuous 
(see \cite[Corollary VIII.9]{Br2}), and integration by part formula is valid. Hence
\begin{equation}\label{sl7}
\sum_{i=1}^{q}  \int_0^T \alpha (t) \theta_{i}(t) \frac{d}{dt} f_{i,\lambda} (u_{\lambda}(t)) dt = - \sum_{i=1}^{q}  \int_0^T \frac{d}{dt}(\alpha \theta_{i})(t)  f_{i,\lambda}(u_{\lambda}(t)) dt.
\end{equation}
By  Lemma \ref{lim06}, 
$$
 f_{i,\lambda}(u_{\lambda})  \rightarrow  f_{i}(u) \ \mbox{uniformly on} \  [0, T], \  \mbox{as}  \  \lambda \rightarrow 0.
$$
Moreover $\frac{d}{dt}(\alpha  \theta_{i}) \in L^{\infty}(0,T; \R)$. Thus, as $\lambda \rightarrow 0$
\begin{equation}\label{sl70}
\sum_{i=1}^{q}  \int_0^T \frac{d}{dt}(\alpha \theta_{i})(t)  f_{i,\lambda}(u_{\lambda}(t)) dt  \rightarrow   \sum_{i=1}^{q}  \int_0^T \frac{d}{dt}(\alpha \theta_{i})(t)  f_{i} (u(t)) dt.
\end{equation}
Since $f_{i} (u)$ is absolutely continuous, using again integration by part formula
\begin{equation}\label{sl71}
- \sum_{i=1}^{q}  \int_0^T \frac{d}{dt}(\alpha \theta_{i})(t)  f_{i} (u(t)) dt =  \sum_{i=1}^{q}  \int_0^T \alpha (t) \theta_{i}(t) \frac{d}{dt} f_i (u(t)) dt.
\end{equation}
From (\ref{sl7}), (\ref{sl70}), and  (\ref{sl71}) we obtain
\begin{equation}\label{sl72}
\sum_{i=1}^{q}  \int_0^T \alpha (t) \theta_{i}(t) \frac{d}{dt} f_{i,\lambda} (u_{\lambda}(t)) dt \rightarrow \sum_{i=1}^{q}  \int_0^T \alpha (t) \theta_{i}(t) \frac{d}{dt} f_i (u(t)) dt.
\end{equation}
Combining (\ref{sl5}), (\ref{sl6}), and (\ref{sl72}) we obtain, for $\theta_{i} $ that satisfy $(\theta_i(t))\in \S^q$,
and are Lipschitz continuous on bounded sets,
\begin{equation}\label{sl8}
 \int_0^T \alpha (t)\| \dot u(t)\|^2 dt   +   \sum_{i=1}^{q}  \int_0^T \alpha (t) \theta_{i}(t) \frac{d}{dt} f_i (u(t)) dt \leq 0.
\end{equation}
Let us show that, by density, (\ref{sl8}) can be extended to arbitrary $\theta_i \in L^{\infty}(0,+ \infty; \R)$ that satisfy $(\theta_i(t)) \in \S^q$. Given such functions $(\theta_i)_{i=1,...,q}$,
by  regularization by convolution,   
we can find a sequence of regular functions $\theta_{i,n} \in {\mathcal C} ^{\infty}(0, + \infty) $,  such that 
\begin{equation}\label{sl80}
\theta_{i,n} \rightarrow \theta_{i} \quad a.e \ \  t  \in (0,+ \infty) \ \text{when } n \text{ goes to } +\infty.
\end{equation}
Let $T:  \mathbb R^q \rightarrow \mathbb R^q$ be the projection onto the unit simplex $\S^q \subset \mathbb R^q$. $T$ is a nonexpansive mapping. By  (\ref{sl80}), and 
$\theta (t)  =(\theta_{i}(t))_i \in \S^q \ \mbox{for almost all} \ t>0$, we see  that  $T\circ \theta_{n}$ is Lipschitz continuous on any interval $[0, T]$,
and satisfies, for almost all $ t>0$
 \begin{align}
 &     T\circ \theta_{n} (t) \in \S^q,\\
 &    (T\circ \theta_{n})_i (t) \rightarrow \theta_{i}(t) \quad \mbox{for almost all} \ t>0. \label{Fatou}
 \end{align}
  By (\ref{sl8}), for each $n \in \mathbb N$, we have
\begin{equation}\label{sl91}
\int_0^T \alpha (t)\| \dot u(t)\|^2 dt + \sum_{i=1}^{q}  \int_0^T \alpha (t) (T\circ \theta_{n})_i(t) \frac{d}{dt} f_{i} (u(t)) dt \leq 0.
\end{equation}
On the other hand, by Lemma \ref{lemma_Brezis}, for any $\xi\in L^2(0, T; H)$ such that $\xi(t)\in \partial f_i(u(t))$ (there exists such elements, for example take  $v_i$ 
obtained in  (\ref{lim112})), we have  
\begin{equation}\label{sl90}
 \frac{d}{dt}  f_i (u(t))=\langle \dot{u}(t), \xi (t)\rangle,
\end{equation}
and $\frac{d}{dt} f_{i} (u)$ is  integrable on $[0,T]$ ($t\mapsto f_i(u(t))$ is absolutely continuous on $[0, T]$). 
From  (\ref{Fatou}), by applying  Fatou's lemma, (note that $\frac{d}{dt} f_{i} (u)\in L^1([0,T])$, which allows us to reduce to the case of non-negative functions), we obtain
\begin{equation}\label{sl92}
\sum_{i=1}^{q}  \int_0^T \alpha (t) \theta_{i}(t) \frac{d}{dt} f_{i} (u(t)) dt \leq  \liminf_n  \sum_{i=1}^{q}  \int_0^T \alpha (t) (T\circ \theta_{n})_i(t)  \frac{d}{dt} f_{i} (u(t)) dt.
\end{equation}
From (\ref{sl91}) and  (\ref{sl92}) we deduce that
\begin{equation}\label{sl93}
\int_0^T \alpha (t)\| \dot u(t)\|^2 dt + \sum_{i=1}^{q}  \int_0^T \alpha (t) \theta_{i}(t) \frac{d}{dt} f_{i} (u(t)) dt \leq 0.
\end{equation}
Since $\alpha$ is an arbitrary  positive test function, we deduce from (\ref{sl93}) that
\begin{equation}\label{sl9}
\| \dot u(t)\|^2 +  \sum_{i=1}^{q}  \theta_{i}(t) \frac{d}{dt} f_{i} (u(t)) dt \leq 0.
\end{equation}
Take arbitrary 
$\eta \in L^2(0,T; \mathcal H), \ \xi_i \in L^2(0,T; \mathcal H) \ i=1,..., q$,  such that
$ \eta(t) \in N_K(u(t)), \  \xi_i(t) \in \partial f_i(u(t)) \ \mbox{for almost all} \ t>0$.
Since $u(t) \in  K$, we have $\dot{u}(t) \in T_K (u(t))$, and since  $ \eta(t) \in N_K(u(t))$
 \begin{equation}\label{sl10}
\left\langle  \dot{u}(t),   \eta(t) \right\rangle \leq 0.
\end{equation}
Combining (\ref{sl90}), (\ref{sl9}), and (\ref{sl10}) we obtain
\begin{equation}\label{sl11}
\| \dot u(t)\|^2 +   \left\langle  \dot{u}(t),   \eta(t) \right\rangle   +  \sum_{i=1}^{q}  \theta_{i}(t) \left\langle  \dot{u}(t),   \xi_i(t) \right\rangle  \leq 0.
\end{equation}
Equivalently, for any $z(t) \in N_K(u(t)) + \mbox{Conv}\left\{\partial f_i(u(t))\right\} $
\begin{equation}\label{sl12}
 \left\langle 0 -  (-\dot{u}(t)),   z(t)  - \dot{u}(t)   \right\rangle  \leq 0.
\end{equation}
Combining this property with (\ref{lim17}) we  obtain
$$
\dot{u}(t) + \bigg( N_K(u(t)) +    \mbox{{\rm Conv}}\left\{ \partial f_i(u(t)) \right\}   \bigg)^{0}     = 0 \quad \mbox{for almost all} \ t>0,
$$
which ends the proof.


 \begin{lemma}\label{L:Rock}
Let $I : L^{\infty}(0,T ; \mathcal H) \longrightarrow \R$ be defined by $I (\xi ) = \int_0^T  \sum_{i=1}^{q} \theta_i (t) f_i (\xi(t)) dt$, with 
 $\theta_i \in L^{\infty}(0,T ; \mathcal H)$ for $i=1,...,q$, and $(\theta_i(t))\in \S^q$ for almost all $t>0$. 
 Let $\ \xi \in L^{\infty}(0,T ; \mathcal H)$ and $z \in L^1(0,T ; \mathcal H)$ such that $z \in \partial I( \xi)$. Then for all $i=1,...,q$ there exists $v_i \in L^{\infty}(0,T ; \mathcal H)$  such that 
\begin{equation*}
\text{for almost all } t>0, \ v_i(t) \in \partial f_i (\xi (t)), \text{  and } \ z(t)={\sum\limits_{i=1}^{q} \theta_i(t) v_i(t)}.
\end{equation*}
\end{lemma}
\begin{proof}

By the Fenchel extremality relation,
 \begin{equation}\label{dual-1}
z \in \partial I (\xi) \Leftrightarrow I(\xi) + I^{*}(z) -\left\langle  \xi,z    \right\rangle_{(L^{\infty}(0,T ; \mathcal H), L^{1}(0,T ; \mathcal H))} = 0.
 \end{equation}
By \cite[Theorem 2]{Rock}, we have
$$
I^{*}(z)=   \int_0^T   \left(\sum_{i=1}^{q} \theta_i (t) f_i \right)^{*} (z(t)) dt.
$$
Let us analyze this last expression. Since the $f_i$ are convex continuous functions, their conjugate are coercive functions, and 
$$
\left(\sum_{i=1}^{q} \theta_i (t) f_i \right)^{*} (z(t))= \min \left(\sum_{i=1}^{q} \left(\theta_i (t) f_i\right)^{*} (z_i): \quad \sum_i z_i= z(t) \right).
$$
The same measurable selection argument as the one used  in Lemma \ref{basic-sel} gives the existence of measurable functions $z_i (\cdot)$ such that
\begin{equation}\label{dual-1.5}
\left(\sum_{i=1}^{q} \theta_i (t) f_i \right)^{*} (z(t))=  \sum_{i=1}^{q} \left(\theta_i (t) f_i\right)^{*} (z_i (t)) \quad \mbox{with}  \  \sum_i z_i (t)= z(t) .
\end{equation}
Returning to (\ref{dual-1}) we obtain 
 \begin{align}\label{dual-2}
z \in \partial I (\xi) &\Leftrightarrow & \int_0^T  \sum_{i=1}^{q} (\theta_i (t) f_i (\xi(t)) + \left(\theta_i (t) f_i\right)^{*} (z_i (t)))dt -\int_0^T \left\langle \xi(t),z(t) \right\rangle dt = 0\\
&\Leftrightarrow & \int_0^T  \sum_{i=1}^{q} \left(\theta_i (t) f_i \right)(\xi(t)) + \left(\theta_i (t) f_i\right)^{*} (z_i (t)) - \left\langle \xi(t),z_i(t)   \right\rangle dt = 0.\nonumber
 \end{align}
Since each of the elements of this last sum expression is nonnegative, we deduce that, for each $i=1,2,...,q$, and for almost all $t>0$
$$
\theta_i (t)f_i (\xi(t)) + \left(\theta_i (t) f_i\right)^{*} (z_i (t)) - \left\langle \xi(t),z_i(t) \right\rangle =0.
$$
Equivalently $z_i (t)  \in \partial \left(\theta_i (t) f_i \right)(\xi(t))$. Let us now verify that 
$$\partial \left(\theta_i (t) f_i \right)(u(t))= \theta_i (t)\partial  f_i (u(t)).$$
Take some $\tilde{z_i} \in L^{\infty}(0,T ; \mathcal H)$ such that $\tilde{z_i}(t) \in \partial f_i (\xi(t)) \quad \mbox{for almost all} \ t>0$, 
(there exists such element, take for example $\tilde{z_i}(t) =(\partial f_i)^{0} (\xi(t))$). 
We  have
 \begin{equation}\label{lim142}
z_i(t)  = \theta_i (t) v_i(t) \quad \mbox{for almost all} \ t>0, 
\end{equation}
where
\begin{equation}
\label{lim143} v_i (t)= \
\begin{cases}
 \frac{z(t)}{\theta_i (t)} \quad \mbox{if  } \theta_i (t) >0, \\
\tilde{z_i}(t) \quad \mbox{if  } \theta_i (t) =0.
\end{cases}
\end{equation}
Moreover $v_i$ is measurable, and $v_i (t) \in \partial  f_i(\xi(t)) \ \mbox{for almost all} \ t>0$. By continuity of  $f_i$, we conclude that
$v_i \in L^{\infty}(0,T ; \mathcal H)$. 
\end{proof}

 \section{Some modeling and numerical aspects, perspectives}\label{mod}

  \subsection{Cooperative games}\label{game}
 
 In this section, we consider some modeling aspects concerning the multiobjective steepest descent for cooperative games.
  This completes  \cite{AG1}, where was considered the smooth case. 
  Indeed, for applications, it is quite useful to consider objective functions which are not differentiable
  (like the $\| \cdot \|_1$ norm  for sparse optimization). 
  
 Let us consider $q$ agents (consumers, social actors, deciders,...). The agent $i$ acts on a decision space $\mathcal H_i$, and takes decision $v_i \in \mathcal H_i$, \ $i=1,2,...,q$. 
 Let $K$ be a given closed subset of $\mathcal H  =  \mathcal H_1 \times \mathcal H_2 \times...\times \mathcal H_q $, which reflects the limitation of ressources, and/or various constraints. 
 Feasible decisions $v \in \mathcal H$ satisfy
 $$
 v= (v_1,v_2,...,v_q) \in   K.
 $$ 
 Each  agent $i$ has a disutility (loss) function  $f_i: \mathcal H \rightarrow \mathbb R$ which associates to each feasible decision $v \in  K$ the scalar $f_i(v)$.
The game in normal form is given by the triplet $(\mathcal H,  K , (f_i)_{i=1,...q} )$.
The (MOG) dynamic has been designed in order to satisfy some desirable properties with respect to Pareto equilibration: each trajectory $t \mapsto u(t)$  of (MOG) satisfies

i) for each $i=1,2,...,q$,  \   $t \mapsto f_i(u(t))$ is  nonincreasing (Theorem \ref{asymp1}, item i)); 

ii)  $u(t)$ converges to a Pareto critical point as $t\rightarrow + \infty$ (Theorem \ref{asymp1}, item iii));
 
Let us make some futher observations:
 
iii)  In (MOG) dynamic there is no a priori or a posteriori  scalarization of the original vector optimization
problem. Neither ordering information nor weighting factors for the
different objective functions are assumed to be known.
The scalarization is done dynamically,  endogenously ((MOG) is an autonomous dynamical system). Taking into account  the worst directional derivative (indeed, in view of minimization, it is  the greatest), 
can make progress all agents, and  gives to (MOG) system robustness  (minimization in the worst case), and good convergence properties.
When it is no longer possible to make progress all the agents, the process stops at a weak Pareto optimal point.
It is a natural question  whether it is possible to reach a Pareto optimum. Indeed, it depends on the willingness of the agents to cooperate more or less.
After reaching a weak Pareto optimum, a natural way is to consider the  coalition  involving agents that can further enhance their performance.
Then we can consider the (MOG) dynamics involving these  agents. An additional constraint must be added which states that the performance of the agents who stay at rest is not damaged.

iv) The choice of the metric on the space $\mathcal H$ plays a fundamental role in the definition of  the gradient-like system (MOG).
The metric reflects  the friction and inertia that are attached to the changes in dynamical decision processes,
 see  \cite{ABRS}, \cite{ASo}, for an account on the notion of costs to change 
(changing a routine...).
The definition of (MOG) involves local notions (subdifferentials of the $f_i$, and tangent cone to $K$) which corresponds to the modeling of myopic agents.

v) A central question in Pareto optimization is obtaining a Pareto optimum with desirable properties. 
A major advantage of the dynamic gradient approach is that we don't need to know the whole Pareto front.
The weak Pareto equilibrium finally reached is not too far from the starting point of the dynamics (see Figures 1 and 2), making the process realistic in engineering and human sciences.
Moreover, one can select a Pareto optimum which is not too far from a desirable state $u_d$ by using an auxiliary asymptotic hierarchical procedure (see \cite{ACz} and references therein).
For example, according to the method of Tikhonov regularization, we can consider $\epsilon (t) \to 0$ as $t\to + \infty$, with $\int_0^{\infty} \epsilon (t) dt = + \infty  $, and the following  dynamics
\begin{center}
	$\dot u(t) + \bigg(N_K(u(t)) + \mbox{Conv}\left\{\partial f_i(u(t)) + \epsilon (t) (u(t) - u_d)\right\} \bigg)^{0}  = 0$.
\end{center}
It is a (time)-multiscaled nonautonomous dynamic, an interesting subject for further research.

vi) Hybrid methods combine gradient methods (fast, with low computational cost, but local) with evolutionary computation methods (global, but with high computational cost).
They  have proved to be efficient for the minimization of a single objective function.  
It would be interesting to develop the same type of idea in order to reach the whole Pareto set,  see  \cite{Bos}, \cite{BrS} for some first results in this direction. 

\subsection{Inverse problems} 

As a model situation, let us consider the  computation of sparse solutions for underdetermined systems of equations. It  is an important problem in signal compression and statistics (see \cite{Don,Tib}). It leads to the following nonsmooth convex minimization problem
$$\min\{   \|  Ax - b \|^2_2    + \alpha    \|x\|_1 :   \quad  x\in K \subset \R^n \}$$
where  $ \|  Ax - b \|^2_2  $ is a  least squares data fitting term, and $ \|x\|_1$ forces sparsity.
There is numerical evidence that a careful weighting of these two terms is important for the effectiveness of the method. Usually it is done by experimental trials.
It would be of great interest to develop a numerical method based on a multiobjective optimization approach (with $f_1 (x) =  \|  Ax - b \|^2_2  $ and $f_2 (x)= \|x\|_1 $), where the weighting is done automatically, while giving more weight to the lower term. 
Indeed, this is what the (MOG) dynamic does. 

 \noindent All these considerations  naturally lead us to consider discretized, algorithmic  versions of the method.

 \subsection {Numerical  descent methods for nonsmooth multiobjective optimization}\label{num}
In the unconstrained case, an explicit discretization of (MOG) provides an algorithm of the form:
\begin{equation}\label{Algo1}
\text{At step } k, \text{ compute } \ u_{k+1} = u_k + \lambda_k d_k,
\end{equation}
where $d_k=s(u_k)$ is the multiobjective steepest descent direction at $u_k$,  and $\lambda_k$ is some nonnegative steplength.
If we have a constraint $K$, we can approach $s(u_k)$  by replacing the tangent cone $T_K(u_k)$ with $\frac{C- u_k}{\mu_k}$ (for some small $\mu_k$) in (\ref{D:SecondForm}). This leads to  :
\begin{align}
 & u_{k+1}=u_k + \lambda_k d_k, \label{Algo2} \\
& \text{where } \ d_k = \underset{d \in C- u_k}{\mbox{argmin}} \left\{ \frac{1}{2 \mu_k} \Vert d \Vert^2 + \max\limits_{i=1,...,q} \max_{p_i\in  \partial f_i(u)}  \left\langle  p_i, d \right\rangle \right\} \label{Algo2bis}.
\end{align}
Note that the algorithms given in (\ref{Algo1}) and (\ref{Algo2}) are equivalent when $K=\mathcal{H}$ and $\mu_k \equiv \mu$.
These algorithms have been studied  in \cite{FS}, \cite{DS} (unconstrained case), in \cite{GI} (constrained case) in a finite-dimensional setting, and assuming that the objective functions are $C^1$ (not necessarily convex). As a distinctive feature of these algorithms, the  steplength $\lambda_k$ is computed by an Armijo-like rule (to secure a descent property), and directions $d_k$ are computed approximatively (with a given tolerance). They lead to the following results:
\begin{enumerate}
\item If $\mu_k \equiv \mu$, then any accumulation point is a critical Pareto point.
\item If the objective functions are convex, and if $\mu_k= \frac{\alpha_k}{\max\limits_{i=1,...,q} \Vert \nabla f_i (u_k) \Vert}$ with $\alpha_k \in \ell^2 \setminus \ell^1$, then any bounded sequence converges to a weak Pareto optimal point.
\end{enumerate}
It appears that these algorithms, which are obtained -at least formally- by the explicit discretization in time of (MOG), share common properties with our continuous dynamic (descent property, convergence to weak Pareto optimal points).
It would be interesting to  justify mathematically  that the continuous and discrete dynamic systems have the same asymptotic behavior, as it was established in the case of a single objective (see \cite{PS}). Another challenging aspect of these algorithms is the effective computation of $d_k$. For instance, in the unconstrained case, we need to solve the minimization problem (\ref{D:FirstForm}), which can be done by applying a  Gauss-Seidel-like method to
\begin{equation}\label{D:Algo2Subroutine}
\underset{ \begin{array}{c} \Lambda=(\lambda_1,...,\lambda_q) \in \S^q \\(p_1,...,p_q)\in \H^q \end{array} }{\text{\rm minimize}} \quad \frac{1}{2} \Vert {\sum\limits_{i=1}^{q} \lambda_i p_i} \Vert^2 + \delta_{\S^p} (\Lambda) + {\sum\limits_{i=1}^{q} \delta_{\partial f_i (u^k)} (p_i)}.
\end{equation}
Problem (\ref{D:SecondForm}) is also well suited for primal-dual methods, and perhaps other methods could be examined and compared. To our knowledge, this work has never been done, and is a subject for further study.

More recently, a trust-region method for unconstrained multiobjective problems involving smooth functions has been developed in \cite{VOS}, which uses the norm of the multiobjective steepest descent vector as a generalized marginal function.
In \cite{FliDruSva09,DRS}, a Newton method for unconstrained strongly convex vector optimization has been developed,  with a local superlinear convergence result.  Instead of taking $d_k$ as a descent direction computed from first-order quadratic models as in (\ref{D:SecondForm}), the authors use  \textit{second-order} quadratic models to define a multiobjective Newton direction as:
\begin{equation}
d_k = \underset{d \in \H}{\text{argmin}} \left\{ \max\limits_{i=1,...,q} \langle \nabla f_i (u_k) , d \rangle + \frac{1}{2} \langle \nabla^2 f_i(u_k) d,d \rangle \right\}.
\end{equation}
As in Theorem \ref{steep;desc;dir}, they show that this discrete dynamic corresponds to the classical Newton's method applied to a weighted combination of the objective functions, but with an endogenous scalarization. In other words, at each step, the algorithm provides $(\theta_i^k) \in \S^q$ such that 
\begin{equation}
d_k = -\left( {\sum\limits_{i=1}^{q} \theta_i^k \nabla^2 f_i (u_k)}  \right)^{-1} \left( {\sum\limits_{i=1}^{q} \theta_i^k \nabla f_i (u_k)} \right).
\end{equation}
See also \cite{Pov14} for second-order models built with a BFGS scheme, to avoid the direct computation of the Hessian. These works suggests the existence of corresponding continuous Newton-like dynamics (see for example \cite{AS} in the
case of a single objective).

 \section{Conclusion, perspectives}\label{cp}
 
In this paper, we have shown some remarkable properties of the multiobjective steepest descent direction, and of the dynamical system which is governed by the corresponding vector field: along each trajectory, all the  objective functions are decreasing, and there is convergence to a weak Pareto minimum. Working in a general Hilbert space, and with convex continuous functions (not necessarily differentiable) allows us to  cover a wide range of applications. However, there are many issues to be resolved. Among the most challenging, let us mention the uniqueness of the solution, for a given Cauchy data, and
the dynamical properties of the weighting coefficients.
The natural link between the (MOG) dynamic and the theory of gradient flows naturally suggests to study the dynamics for semi-algebraic functions, on the basis of  Kurdyka-Lojasiewicz inequality.
Obtening rapid methods based on an analysis of second order in time (inertial aspects), or space (Newton-like methods) is important both from the numerical, and modeling point of view.
It would be also interesting to consider interior point methods.
Some modeling aspects in game theory, economics, and inverse problems, have been considered in the previous sections. They are still largely unexplored.
All these  results suggest that there is a broad class of continuous dynamics that contains (MOG), and having similary properties with respect to Pareto equilibration.  Enriching this class of dynamics can be useful for numerical purpose, and  for understanding the complex interactions in Pareto equilibration (coalitions, negotiation, bargaining, dealing with uncertainty, changes in the environment, psychological aspects). 
These are interesting topics for further research.

\if{
\newpage

\noindent $\mathcal H= \mathbb R \times \mathbb R $, \ $u = (x,y)$.

\bigskip

\noindent 1. \ $f_1(x,y)= \frac{1}{2}x^2$  \quad  \quad  $f_2(x,y)= \frac{1}{2}y^2$

\bigskip

\noindent $s(x,y) = - (\mbox{{\rm Conv}}\left\{\nabla f_i(x,y) \right\} )^{0}=  -\left(  \frac {xy^2}{x^2 + y^2}, \   \frac {yx^2}{x^2 + y^2}        \right)$

\medskip

\noindent Pareto set:  \ $\mathbb R \times \left\{0\right\} \cup   \left\{0\right\} \times \mathbb R$

\medskip

\setlength{\unitlength}{7cm}
\begin{picture}(0.5,0.7)(-0.6,0)
\put(-0.04,0.01){$0$}

\put(0,0){\line(1,1){0.65}}
\put(0.65,0.6){$x=y$}	
\put(0.8,0.41){$M  (x,y)$}
\put(0.8,0.4){\line(0,-1){0.4}}
\put(0.8,0.4){\line(-1,0){0.8}}	
\put(0,0.4){\line(2,-1){0.8}}	
\put(0.8,0.4){\vector(-1,-2){0.16}}
\linethickness{0.3mm}
\put(0.785,-0.04){$x$}
\put(-0.05,0.385){$y$}
\put(0.52,0.2){$s(x,y)$}
\put(0.0,-0.1){\vector(0,1){0.8}}	
\put(-0.1,0){\vector(1,0){1.01}}

\qbezier(0.69,0)(0.71,0.155)(0.8,0.4)

\qbezier(0.45,0.43)(0.2,0.15)(0.19,0.135)
\qbezier(0.19,0.135)(0.18,0.12)(0.17,0.1065)
\qbezier(0.17,0.1065)(0.16,0.09)(0.15,0.07)
\qbezier(0.15,0.07)(0.14,0.045)(0.135,0.025)
\qbezier(0.135,0.025)(0.133,0.0115)(0.1326,0)

\qbezier(0.0,0.415)(0.17,0.45)(0.5,0.65)
\qbezier(0.0,0.345)(0.17,0.385)(0.55,0.65)
\qbezier(0.0,0.25)(0.17,0.3)(0.6,0.65)
\end{picture}

\bigskip

\begin{center}
Figure 1
\end{center}

\bigskip

\bigskip

\bigskip 

\noindent 2.  \  $f_1(x,y)= \frac{1}{2}(x+1)^2    + \frac{1}{2}(y)^2 $  \quad  \quad  $f_2(x,y)= \frac{1}{2}(x-1)^2    + \frac{1}{2}(y)^2 $ 

\bigskip

\begin{equation*}
s(x,y)=
\begin{cases}
 -(x-1,y)     \  \mbox{if } \  x>1,\\
 -(0,y) \  \mbox{if } \  -1 \leq x \leq 1,\\
 -(x+1,y)     \  \mbox{if } \   x<-1
\end{cases}
\end{equation*}

Pareto set: $[-1, +1] \times \left\{0\right\}$.

\setlength{\unitlength}{7cm}
\begin{picture}(0.3,0.7)(-0.6,-0.08)

\put(1.25,0.02){$x$}
\put(0.38,0.56){$y$}
\put(0.37,0.02){$0$}
\put(0.75,0.02){$+1$}
\put(0.02,0.01){$-1$}

\put(0.923,0.45){\vector(-1,-2){0.03}}
\put(-0.198,0.4){\vector(1,-2){0.03}}
\put(0.5,0.4){\vector(0,-1){0.05}}

\put(0.7,-0.1){\line(0,1){0.7}}
\put(0.0,-0.1){\line(0,1){0.7}}

\put(0.35,-0.1){\vector(0,1){0.7}}	
\put(-0.3,0){\vector(1,0){1.6}}
\put(0.7,0.00){\line(1,2){0.25}}
\put(0.0,0.00){\line(-1,2){0.25}}
\put(0.5,0.00){\line(0,1){0.5}}
\put(0.06,0.00){\line(1,0){0.58}}   

\linethickness{0.4mm}
\put(0.0,0.0){\line(1,0){0.7}}
\end{picture}

\bigskip

\begin{center}
Figure 2
\end{center}

\bigskip

\bigskip

\newpage
\bigskip

\bigskip

3. Example where the vector field $u\mapsto s(u)$ is not Lipschitz.

\bigskip

\begin{center}
$\|  \mbox{{\rm proj}}_C (v) -  \mbox{{\rm proj}}_D (v)   \| \leq   \rho(\|v\|) \mbox{{\rm haus}}(C,D)^{\frac{1}{2}}$
\end{center}

\setlength{\unitlength}{6cm}

\begin{picture}(0.6,0.6)(-0.35,0)
\put(-0.02,0.02){$0$}

\put(0.138,0){\line(-1,3){0.02}}

\put(0.145,0.010){$\theta$}
\put(1.01,0){$A$}
\put(1.01,0.25){$D$}
\put(0.82,0.20){$C$}
\put(0.96,0.15){$\theta$}


\put(1,0.133){\line(-3,-1){0.06}}

\put(0.95,0){\line(0,1){0.05}}
\put(1,0.05){\line(-1,0){0.05}}
\put(1.005,0.505){$H$}
\put(0.75,0.41){$K$}
  
\put(0.737,0.37){\line(1,-2){0.03}}

\put(0.764,0.31){\line(2,1){0.065}} 
\put(0,0){\line(2,1){1}}
\put(0,0){\line(1,0){1}}
\put(0.5,0.01){$1$}

\put(0.95,0){\line(0,1){0.05}}
\put(1,0.05){\line(-1,0){0.05}}
\put(1,0){\line(-1,2){0.2}}	
\put(0.997,0){\line(0,1){0.5}}
\end{picture}

\bigskip

\begin{itemize}
\item $f_1 (x,y) =\frac{1}{2}(x^2 + y^2)$, \  $f_2 (x,y)= x   $;  \\  $\nabla f_1 (x,y) =(x,y)$, \quad \ \   $\nabla f_2 (x,y) =(1,0)$

\medskip

\item $-s(K) = \vec{OK}$; \ $-s(H) = \vec{OA}$

\medskip

\item $    \| s(H) - s(K)    \|   = \|AK\|   = \sin (\theta) \approx \theta $  \quad ($\theta$ small)

\medskip

\item $   \| HK\| = \tan (\theta)   \sin (\theta )   \approx \theta^2 $

\end{itemize}

\begin{center}
Figure 3
\end{center}

\bigskip

\vspace{3cm}
\setlength{\unitlength}{3cm}

\begin{picture}(0.5,0.5)(-1.5,0)

\put(-0.08,-0.1){$0$} 

\put(0.0,-0.5){\vector(0,1){2}} 	
\put(-1,0.0){\vector(1,0){2.5}} 

\put(1,-0.5){\line(0,2){2}} 
\put(1.05,-0.5){$x=1$}

\put(0.5,0){\circle{1}}
\put(1.01,0.05){$(1,0)$}
\put(1,0){\circle*{0.04}}

\linethickness{0.5mm}

\qbezier(1.5,1.4)(1.25,1.4)(1,1.4)
\qbezier(1,1.4)(0.9,1.398)(0.8,1.39)
\qbezier(0.8,1.39)(0.7,1.37)(0.61,1.35)
\qbezier(0.61,1.35)(0.51,1.32)(0.43,1.29)
\qbezier(0.43,1.29)(0.34,1.25)(0.23,1.18)
\qbezier(0.23,1.18)(0.1,1.09)(0,1)
\qbezier(0,1)(-0.1,0.89)(-0.17,0.8)
\qbezier(-0.17,0.8)(-0.23,0.7)(-0.28,0.6)
\qbezier(-0.28,0.6)(-0.33,0.51)(-0.36,0.4)
\qbezier(-0.36,0.4)(-0.39,0.3)(-0.41,0.2)
\qbezier(-0.41,0.2)(-0.42,0.1)(-0.425,0)

\qbezier(1.5,1)(1.25,1)(1,1)
\qbezier(1,1)(0.9,0.99)(0.8,0.985)
\qbezier(0.8,0.985)(0.7,0.963)(0.61,0.93)
\qbezier(0.61,0.93)(0.53,0.897)(0.42,0.83)
\qbezier(0.42,0.83)(0.32,0.76)(0.22,0.66)
\qbezier(0.22,0.66)(0.148,0.56)(0.065,0.415)
\qbezier(0.065,0.415)(0.04,0.36)(0,0.22)
\qbezier(0,0.22)(-0.01,0.175)(-0.025,0)

\qbezier(0,1.4)(-0.09,1.328)(-0.2,1.22)
\qbezier(-0.2,1.22)(-0.3,1.13)(-0.345,1.07)
\qbezier(-0.345,1.07)(-0.4,1)(-0.5,0.847)
\qbezier(-0.5,0.847)(-0.56,0.74)(-0.6,0.65)
\qbezier(-0.6,0.65)(-0.639,0.55)(-0.669,0.45)
\qbezier(-0.669,0.45)(-0.69,0.35)(-0.71,0.25)
\qbezier(-0.71,0.25)(-0.72,0.15)(-0.73,0)

\qbezier(1.5,0.6)(1.25,0.6)(1,0.6)
\qbezier(1,0.6)(0.9,0.592)(0.8,0.57)
\qbezier(0.8,0.57)(0.7,0.525)(0.64,0.48)
\qbezier(0.64,0.48)(0.32,0.24)(0,0)

\qbezier(1.5,0.4)(1.25,0.4)(1,0.4)
\qbezier(1,0.4)(0.95,0.397)(0.9,0.389)
\qbezier(0.9,0.389)(0.85,0.375)(0.84,0.368)
\qbezier(0.84,0.368)(0.42,0.184)(0,0)

\qbezier(1.5,-0.2)(1.25,-0.2)(1,-0.2)
\qbezier(1,-0.2)(0.98,-0.198)(0.96,-0.196)
\qbezier(0.96,-0.196)(0.48,-0.098)(0,0)

\end{picture}

\vspace{1.5cm}

}\fi

\end{document}